\documentclass[11pt]{article}
\usepackage{amsmath, amssymb, amsfonts, amsthm, authblk, color, graphicx, enumitem, mathrsfs, indentfirst}
\usepackage{hyperref}
\hypersetup{colorlinks=true, linkcolor=blue, filecolor=magenta, urlcolor=cyan}
\usepackage[textwidth=6in,textheight=9in]{geometry}
\allowdisplaybreaks

\newtheorem{theorem}{Theorem}
\newtheorem{lemma}{Lemma}[section]

\newtheorem{remark}{Remark}[section]

\newtheorem{problem}{Problem}
\newtheorem{corollary}{Corollary}[section]
\newtheorem{example}{Example}
\newtheorem{algorithm}{Algorithm}

\numberwithin{equation}{section}

\newcommand{\keywords}[1]{\small\textbf{\textit{Keywords---}}#1}

\title{Scattered point measurement-based regularization for backward problems for fractional wave equations}

\author[1]{Dakang Cen}
\author[2]{Zhiyuan Li\footnote{Corresponding author 1: lizhiyuan@nbu.edu.cn, supported by the National Natural Science Foundation of China (no. 12271277), Ningbo Youth Leading Talent Project (no. 2024QL045).  and the Open Research Fund of the Key Laboratory of Nonlinear Analysis \& Applications (Central China Normal University), Ministry of Education, China (no. NAA20230RG002).}}

\author[3]{Wenlong Zhang\footnote{Corresponding author 2: zhangwl@sustech.edu.cn, supported by the National Natural Science Foundation of China under grant
numbers No.12371423 and No.12241104.}}
\affil[1,3]{Department of Mathematics, Southern University of Science and Technology, Shenzhen, 518055, China}
\affil[2]{School of Mathematics and Statistics, Ningbo University, Ningbo, 315211, China}

\begin{document}
\maketitle

\abstract{In this work, we are devoted to the reconstruction of an unknown initial value from the terminal data. The asymptotic and root-distribution properties of Mittag-Leffler functions are used to establish stability of the backward problem. Furthermore, we introduce a regularization method that effectively handles scattered point measurements contaminated with stochastic noise. Furthermore, we prove the stochastic convergence of our proposed regularization and provide an iterative algorithm to find the optimal regularization parameter. Finally, several numerical experiments are presented to demonstrate the efficiency and accuracy of the algorithm.}



\section{Introduction}
Assuming that $\alpha\in(1,2)$ and that $\Omega \subset \mathbb{R}^d$, $d=1,2,3$, is a bounded domain with sufficiently smooth boundary $\partial \Omega$, we consider the following fractional wave equation:
\begin{equation}\label{eq-gov}
\begin{cases}
  \partial_t^\alpha u - \Delta u = 0, & (x,t) \in \Omega \times (0,T), \\
  u(x,0) = a_0(x), & x \in \Omega, \\
  \frac{\partial}{\partial t}u(x,0) = a_1(x), & x \in \Omega, \\
  u(x,t) = 0, & (x,t) \in \partial \Omega \times (0,T),
\end{cases}
\end{equation}
where the operator $\partial_t^\alpha$ is referred to as the Caputo derivative of order $\alpha$, defined by
$$
\partial_t^\alpha \psi(t) := \frac{1}{\Gamma(2 - \alpha)} \int_0^t (t-\tau)^{1-\alpha} \psi''(\tau) d\tau, \quad t>0.
$$
For $\alpha$ within the interval $(1,2)$, the equation is classified as a fractional wave equation or diffusion wave equation, which enables more accurate simulation of diffusion processes in media exhibiting memory effects. For example, these equations are particularly adept at describing nonlocal interactions, non-Markovian dynamics, and non-Gaussian statistical behaviors, which are common in various physical, biological, and financial systems, see, e.g., \cite{hatano1998dispersive}, \cite{Kilbas2006theory}, \cite{metzler2000random}, \cite{podlubny1999fractional} and the references therein.

Their significance has not only been widely recognized in the physics, but has also received much attention from mathematicians. Given many related mathematical analyses, we do not intend to give a full list of references. Regarding forward problems for \eqref{eq-gov}, we refer to Kubica et al. \cite{Kubica2020intro}, Sakamoto and Yamamoto \cite{sakamoto2011initial}, and Zacher \cite{Zacher2009weak}. Inverse problems for fractional differential equations have also been investigated recently. We recommend consultation on research on different kinds of inverse problems, e.g., inverse coefficient problems of Li et al. \cite{liliuyamamoto2019inverse}, inverse parameter problems of Li and Yamamoto \cite{liyamamoto2019inverse} and inverse parameter problems of Liu et al. \cite{liu2019inverseproblems}.

Among these areas, backward problems are particularly important as they aim at detecting the previous status of the physical field from its terminal observational data, which is of importance for engineering to understand and control the systems. In this paper, we aim to recover the initial value $a_1$ using the terminal data $u(x,T), x \in \Omega$.
\begin{problem}%\label{prob-isp}
Letting $a_0=0$ in $\Omega$, we take the terminal value $u(\cdot,T)$ as the observation data to determine the unknown initial value $a_1$. We propose:
\begin{enumerate}
\item $($Uniqueness$)$ Whether the measurement uniquely determines $a_1(x)$, $x\in\Omega$.
\item $($Stability$)$ How to establish stability for the mapping $u(x,T) \mapsto a_1(x)$?
\item $($Reconstruction$)$ How to design an efficient algorithm to recover the unknown initial value $a_1$ from the scattered point-wise measurement in a practical experiment?
\end{enumerate}
\end{problem}

The backward problems for fractional differential equations have been extensively studied in recent years. We do not intend to provide a comprehensive list of references. Sakamoto and Yamamoto \cite{sakamoto2011initial} were the first to establish the well-posedness of a backward problem by using the eigenfunction expansion and the properties of Mittag-Leffler functions. After this, there have been many theoretical and numerical works on backward problems. Floridia et al. \cite{floridia2020backward} considered a fractional diffusion equation with non-symmetric elliptic operator and proved the well-posedness for the backward problem by the completeness of the eigenfunctions of generalized elliptic operators. Chorfi et al. \cite{Chorfi2024backward} and \cite{Chorfi2022log} extended the logarithmic convexity technique to the fractional framework and proved conditional stability estimates of H\"older type for the backward problem. Tuan et al. \cite{Tuan2019backward} proved the unique existence of a local solution to the backward problem for a fractional diffusion with a nonlinear source. For numerical treatment and applications, Liu and Yamamoto \cite{Liu2010backward} proposed an efficient regularization scheme by quasi-reversibility with theoretical analysis. We also refer to Wang and Liu \cite{Wang2013total} where the backward problem for the fractional diffusion equation was used to formulate the deblurring process in image restoration, while the backward time-fractional diffusion problem was solved by using total variation regularization. Wang et al. \cite{Wang2013tikhonov} was devoted to solving a backward problem for a time-fractional diffusion equation with variable coefficients in a general bounded domain, and the convergence rates for the Tikhonov regularized solution were established. Recently, Jin et al. \cite{Jin2023inverse} presented uniqueness and stability results for a backward problem from terminal observation at an unknown time. Sun et al. \cite{Sun2023fractional} introduced a fractional-order quasi-reversibility regularization method to solve a backward problem for an anomalous diffusion model with multi-term time fractional derivatives and gave the convergence rate under the a priori regularization parameter choice rule. Ren et al. \cite{Ren2023regula} proposed a projection regularization method and obtained a uniform error estimate with an optimal convergence rate. For backward problems to time-space fractional diffusion equations, one can refer to Feng et al. \cite{Feng2024numerical}, Jia et al. \cite{Jia2018backward} and the references therein.

However, to the best of the authors' knowledge, despite the case where $1<\alpha<2$ is used in modeling, there are still no works on the backward problem specifically addressing this range, except for those by He and Zhou \cite{He2022backward}, Shi et al. \cite{Shi2023backward}, Wei and Zhang \cite{Wei2018backward}, Wen et al. \cite{Wen2023solving}, Zhang and Zhou \cite{Zhang2022backward} and Floridia and Yamamoto \cite{floridia2020backward}. In these works, the primary challenges arise because, in the case of $1<\alpha<2$, the Mittag-Leffler functions associated with the solution to the equation \eqref{eq-gov} exhibit behaviors that are fundamentally different from those observed in the case of $0<\alpha<1$, e.g., the Mittag-Leffler function in the case of $\alpha\in(1,2)$ exhibits a distinct property, the existence of real zeros, which is markedly different from the case when $0 < \alpha < 1$. The potential existence of real roots makes the solution to the backward problem non-unique, which presents a challenge in the analysis of the diffusion-wave case. We must make additional assumptions on the terminal time, initial value, or observation data. For example, Floridia and Yamamoto \cite{floridia2020backward} constructed a stable framework for solving the backward problem by excluding specific terminal values $T$, Zhang and Zhou \cite{Zhang2022backward} simultaneously recovered two initial values from terminal observations at two sufficiently large time levels. Later, Wen et al. \cite{Wen2023solving} proved the uniqueness of the backward problem with the aid of additional measurement data at two sufficiently fixed close times. Shi et al. \cite{Shi2023backward} and Wei and Zhang \cite{Wei2018backward} incorporated prior information into the unknown initial value and measurement data, respectively, deriving the stability of the backward problem.

In this work, we focus on solving backward problems for problem \eqref{eq-gov}, particularly when the observation points are scattered with noise in $\Omega$. This approach differs from most existing regularization theories which are typically based on full-domain observations. Here we list the key challenges and the main innovations. The key challenges are as follows.
\begin{enumerate}
    \item In the case of $\alpha \in (1, 2)$, the potential existence of real zeros in the Mittag-Leffler function induces oscillatory behavior in solutions, which significantly complicates the stability analysis of inverse problems for fractional wave equations.

    \item Traditional regularization methods cannot analyze the impact of the number of observation points on inversion results based on the theoretical stability results.
\end{enumerate}
Here are the major innovations:
\begin{enumerate}
    \item For $1 < \alpha < 2$, by noting that the Mittag-Leffler function possesses at most finitely many zeros, we avoid zero-induced instabilities by selecting a sufficiently large observation time $T$.  Notably, when $\alpha$ is in $(1, \frac{4}{3}]$, additional constraints on the terminal time $T$ are not required, which relaxes the conditions for the stability of the backward problem to the fractional wave equations.
    \item We treat scattered point measurement as fully stochastic variables, which more realistically reflect the characteristics of observational data. We then propose a Tikhonov regularization method based on scattered observations and prove new optimal error estimates, which indicate the explicit dependency on key parameters such as the noise level, the regularization parameter, and the number of observation points.
\end{enumerate}

The remaining parts of this paper are organized as follows. In Section \ref{sec-pre}, we will review the basic theories of fractional differential equations and Mittag-Leffler functions. In Section \ref{sec-proof}, we will expand on our regularization method and provide proofs to support its effectiveness. Subsequently, we will carry out a series of numerical experiments to verify the performance of our algorithm and demonstrate its robustness under different noise levels and initial conditions in Section  \ref{sec-num}. In conclusion, we present a summary of our findings and outline potential strategies for future research.

\section{Preliminaries}\label{sec-pre}
We introduce $E_{\alpha,\beta}(z)$ which is known as the Mittag-Leffler function of two parameters $\alpha>0$ and $\beta\in\mathbb C$, defined by
\begin{equation*}
  E_{\alpha,\beta}(z) = \sum_{k=0}^{\infty} \frac{z^k}{\Gamma(\alpha k + \beta)}, \quad z \in \mathbb{C}.
\end{equation*}
We have several important properties of the Mittag-Leffler functions.  The first one is related to the asymptotic expansion of the Mittag-Leffler function as follows:
\begin{lemma}(p. 34--35 in \cite{podlubny1998fract})\label{lem-ml-asymp}
If $0<\alpha<2$, $\beta$ is an arbitrary complex number and $\mu$ is an arbitrary real number such that
$$
  \frac{\pi\alpha}2 <\mu<\min\{\pi,\pi\alpha\},
$$
then for an arbitrary integer $p\ge1$ the following expansion holds:
$$
  E_{\alpha,\beta}(z)
= -\sum_{k=1}^p \frac{z^{-k}}{\Gamma(\beta-\alpha k)} +O(|z|^{-1-p}),\quad \text{as }  |z|\to\infty, \ \mu \le |\arg z|\le\pi.
$$
Moreover, the Mittag-Leffler function $E_{\alpha,\beta}(z)$ admits the upper estimate:
$$
|E_{\alpha,\beta}(z)| \le \frac{C}{1+|z|} ,\quad  \mu\le |\arg z|\le\pi.
$$
\end{lemma}
If $\alpha\in(1,2)$, the Mittag-Leffler functions $E_{\alpha,\beta}(-t)$ may have zero points on the real axis, which means they cannot have strictly positive lower bounds. To obtain a positive lower bound, one possible way is to constrain the range of the variable $t$.

Again, from the asymptotic expansion, we can assert that there is at least a finite number of real roots of the Mittag-Leffler function $E_{\alpha,\beta}(-t)$, so we assume $t_1, t_2, \cdots, t_N$ to be the zero points of the function $E_{\alpha,\beta}(-t)$. Then we have
\begin{lemma}\label{lem-ml2-lower}
Given $1< \alpha < 2$, then the following estimate
\[
|E_{\alpha, \beta}(-\lambda T^{\alpha})| \geq \frac{C_T}{1+\lambda T^\alpha}
\]
is valid for any $T^\alpha \not\in \{\frac{t_k}{\lambda}\}_{k=1}^N$, where the constant $C_T>0$ only depends on $\alpha$, $\beta$ and $T$.
\end{lemma}
The proof can be found in Floridia and Yamamoto \cite{floridia2020backward}. Unfortunately, the lower bound above for the Mittag-Leffler function $E_{\alpha,\beta}$ is generally not true for any $T$ and $\lambda>0$ because of the real roots.
However, if restricting $\alpha$ to the range $(1,\frac43]$ and letting $\beta=2$, we can assert that the Mittag-Leffler function $E_{\alpha,2}(z)$ has no real roots.
\begin{lemma}\label{lem-Ea2}
If $ 1 < \alpha \leq \frac{4}{3} $, then $ E_{\alpha, 2}(t) > 0 $ for all $ t \in \mathbb{R} $.
\end{lemma}
\begin{proof}
From Theorem 2 in \cite{pskhu2005real}, we have $
D_{2, 3}^* \subset B^*$, where the sets $D_{2,3}^*$ and $B^*$ are defined by
\begin{align*}
 D_{2, 3}^* & := \{ (\xi, \eta) : 0 < \xi \leq 2, \; \eta \geq \frac{3}{2}\xi, \; (\xi, \eta) \neq (2, 3) \}, \\
 B^* & := \{ (\xi, \eta) : \xi, \eta > 0, \; \forall t \in \mathbb{R}, \; E_{\xi, \eta}(t) > 0 \}.
\end{align*}
For $ 1 < \alpha \le\frac{4}{3} $, it is easy to verify that $ (\alpha, 2) \in D_{2, 3}^* $. Thus, for all $ t \in \mathbb{R} $, we have $ E_{\alpha, 2}(t) > 0 $. We complete the proof of the lemma.
\end{proof}
Based on this positivity property of the Mittag-Leffler function in the above lemma, we can have a lower bound of $E_{\alpha,2}(\cdot)$ in the case of $\alpha\in(1,\frac43]$, that is,
\begin{corollary}\label{coro-ml-lower43}
If $\lambda > 0$, then for $\alpha \in (1,\frac{4}{3} ]$, we have that
\[
E_{\alpha, 2}(-\lambda t^{\alpha}) \geq \frac{C}{1 + \lambda t^\alpha}
\]
is valid for any $t>0$, where the constant $C>0$ depends only on $\alpha$.
\end{corollary}
\begin{proof}
If $0 \leq \lambda t^\alpha \leq M$, where $ M >0 $ is a sufficiently large constant, by Lemma \ref{lem-Ea2} and the continuity of the Mittag-Leffler functions, we have
\[
E_{\alpha, 2}(-\lambda t^{\alpha}) \geq C_0 > 0.
\]
In the case of $ \lambda t^{\alpha} \geq M $, by Lemma \ref{lem-ml-asymp}, taking $ p = 1 $, we can easily see that
\[
E_{\alpha, 2}(-\lambda t^{\alpha}) = \frac{1}{\Gamma(2 - \alpha)}\cdot\frac{1}{\lambda t^{\alpha}} + O(\frac1{\lambda^2 t^{2\alpha}}).
\]
Therefore,
\[
E_{\alpha, 2}(-\lambda t^{\alpha}) \geq \frac{1}{2 \Gamma(2 - \alpha)}\cdot \frac{1}{\lambda t^{\alpha}}
\]
is valid for $\lambda t^{\alpha} \geq M$ with sufficiently large $M>0$. Finally, collecting all the above estimates, we obtain
\[
E_{\alpha, 2}(-\lambda t^{\alpha}) \geq \frac{C}{1 + \lambda t^{\alpha}}, \quad t>0,
\]
where the constant $C>0$ depends only on $\alpha$. This completes the proof of the lemma.
\end{proof}

Collecting all the above estimates in Lemma \ref{lem-ml2-lower}, and Corollary \ref{coro-ml-lower43}, we have the upper and lower estimates for the Mittag-Leffler functions.
\begin{corollary}\label{coro-<ml<}
Let $\lambda,T>0$, and let $\{t_k\}_{k=1}^N$ be zero points of $E_{\alpha,2}(-t)$ if exists, then
\begin{enumerate}
    \item If $\alpha \in (1,\frac43]$, we have
\[
\frac{C^{-1}}{1 + \lambda T^\alpha} \le E_{\alpha, 2}(-\lambda T^{\alpha}) \leq \frac{C}{1 + \lambda T^\alpha}.
\]
\item If $\alpha\in (\frac43,2)$, we have
\[
\frac{C_T^{-1}}{1 + \lambda T^\alpha} \le |E_{\alpha, 2}(-\lambda T^{\alpha})| \leq \frac{C_T}{1 + \lambda T^\alpha},\quad T^\alpha\not\in \big\{\tfrac{t_k}{\lambda}\big\}_{k=1}^N,
\]
\end{enumerate}
where the constant $C>0$ only depends on $\alpha$, and constant  $C_T$ depends on $\alpha$ and $T$.
\end{corollary}

\subsection{Forward problem}
In this section, we recall the well-posedness result of the initial-boundary value problem \eqref{eq-gov}. For this, we state the analytic setting. Let $L^2(\Omega)$ be the square-integrable function space with inner product $(\cdot,\cdot)_{L^2(\Omega)}$ (or $(\cdot,\cdot)$ for short) and let $H^1(\Omega)$, $H^2(\Omega)$ etc. be the usual Sobolev spaces.

The set $\{\lambda_k, \varphi_k\}_{k=1}^{\infty}$ constitutes the Dirichlet eigensystem of the elliptic operator $-\Delta: H^2(\Omega) \cap H_0^1(\Omega) \to L^2(\Omega)$, specifically,
\begin{equation}\label{eigensys}
\begin{cases}
-\Delta \varphi_k = \lambda_k \varphi_k & \text{in } \Omega, \\
\varphi_k =0 & \text{on } \partial \Omega,
\end{cases}
\end{equation}
where $\lambda_k$ is the eigenvalue of the operator $-\Delta$ and satisfies $0 < \lambda_1 \leq \lambda_2 \leq \ldots , \lambda_k \rightarrow \infty$ as $k \rightarrow \infty$, and $\varphi_k$ is the eigenfunctions corresponding to the value $\lambda_k$ and $\{\varphi_k\}_{k=1}^\infty$ forms an orthonormal basis in $L^2(\Omega)$. We have the asymptotic behavior of the eigenvalue $\lambda_k\sim k^{2/d}$ as $k\to\infty$. Then for $\gamma\in\mathbb R$, fractional power $(-\Delta )^{\gamma}$ can be defined
\begin{equation*}
  (-\Delta )^{\gamma}\psi:=\sum_{k=1}^{\infty}\lambda_k^{\gamma}(\psi, \varphi _k)\varphi _k,\quad \psi \in D((-\Delta)^\gamma),
\end{equation*}
where
$$\mathcal{D}((-\Delta )^{\gamma}):=\left \{ \psi\in L^{2}(\Omega); \sum_{k=1}^{\infty}\left |\lambda_k^{\gamma}(\psi, \varphi _k ) \right |^2<\infty  \right \}.
$$
The space $D((-\Delta)^\gamma)$ is a Hilbert space equipped with the inner product
\begin{equation*}
( \psi, \phi)_{D((-\Delta)^\gamma)}  = \left ((-\Delta)^\gamma \psi, (-\Delta)^\gamma\phi \right )_{L^2(\Omega)}.
\end{equation*}
Moreover, we define the norm
$$
\begin{aligned}
\left \| \psi  \right \|_{\mathcal{D}((-\Delta )^\gamma) }
&= \left ((-\Delta)^\gamma \psi, (-\Delta)^\gamma\psi \right )_{L^2(\Omega)}^{\frac12}
= \left ( \sum_{n=1}^{\infty}\left |\lambda_n^{\gamma}(\psi, \varphi_n )  \right |^2 \right )^\frac{1}{2}.
\end{aligned}
$$
For short, we also denote the inner product $(\cdot,\cdot)_{\mathcal D((-\Delta)^\gamma)}$ and the norm $\|\cdot\|_{\mathcal D((-\Delta)^\gamma)}$ as $(\cdot,\cdot)_\gamma$ and $\|\cdot\|_\gamma$ if no confusion arises. Furthermore, it satisfies $\mathcal{D}((-\Delta )^\gamma)\subset{H^{2\gamma}(\Omega)}$ for $\gamma>0$. In particular, we have $\mathcal{D}((-\Delta )^\frac{1}{2}) = H_0^1(\Omega)$, $\mathcal{D}((-\Delta )^{-\frac{1}{2}}) = H^{-1}(\Omega)$ and the norm equivalence $\left \| \cdot   \right \| _{\mathcal{D}((-\Delta )^\gamma  ) }\sim \left \| \cdot  \right \| _{H^{2\gamma}(\Omega)}$ with $\gamma=\pm \frac12$.

In view of the paper of Sakamoto and Yamamoto \cite{sakamoto2011initial}, we have the wellposedness result for the problem \eqref{eq-gov}.
\begin{lemma}\label{lem-forward}
Letting $a_0,a_1\in L^2(\Omega)$, then the initial-boundary value problem \eqref{eq-gov} admits a unique weak solution $u\in L^2(0,T;H^2(\Omega)\cap H_0^1(\Omega))$ and there exists a positive constant $C$ which only depends on $\alpha,\Omega$ such that
$$
\|u(\cdot, T)\|_{H^2(\Omega)} \le C \sum_{j=0}^1 T^{j-\alpha} \|a_j\|_{L^2(\Omega)}.
$$
\end{lemma}

\section{Backward problem for diffusion wave equation}\label{sec-proof}
In this part,  we assume $a_0=0$ and consider only the inversion of the unknown initial value $a_1(x)$ using the terminal data.
\begin{theorem}\label{thm-isp}
Letting $\alpha\in(1,2)$ and $\gamma\in\mathbb R$, we assume $a_1\in L^2(\Omega)$ in \eqref{eq-gov}, then
\[
\|a_1\|_{D((-\Delta)^{-\gamma})} \leq C_T \|u(\cdot,T)\|_{D((-\Delta)^{1-\gamma})}
\]
holds true for $T>0$ satisfying $T^\alpha \not\in \cup_{k=1}^\infty \{ \frac{t_1}{\lambda_k}, \cdots, \frac{t_N}{\lambda_k}\}$, where $\{t_k\}_{k=1}^N$ is the set of positive roots of the Mittag-Leffler function $E_{\alpha,2}(-t)$ and the constant $C_T>0$ depends only on $\alpha,\gamma,d,\Omega,T$. Moreover, if we restrict the order $\alpha\in (1,\frac43]$, then there exists a positive constant $C=C(\alpha,\gamma,d,\Omega)$ such that
\[
\|a_1\|_{D((-\Delta)^{-\gamma})} \leq C T^{\alpha - 1} \|u(\cdot,T)\|_{D((-\Delta)^{1-\gamma})}.
\]
\end{theorem}
\begin{proof}
According to the results in Lemma \ref{lem-forward}, the solution to the problem \eqref{eq-gov} can be expressed as:
\begin{equation}\label{eq-u-sub}
  u(x,t) = \sum_{n=1}^{\infty} t E_{\alpha,2}(-\lambda_n t^\alpha)(a_1, \varphi_n) \varphi_n(x).
\end{equation}
Here, $\{\lambda_n,\varphi_n\}$ is the eigensystem defined in \eqref{eigensys}. Taking $t=T$ in \eqref{eq-u-sub} and taking the inner product with $\varphi_k$ and integrating over $\Omega$, we obtain
\[
(u(\cdot, T), \varphi_k) = TE_{\alpha, 2} (-\lambda_k T^{\alpha}) (a_1, \varphi_k),\quad k=1,\cdots,
\]
from which we solve for $(a_1, \varphi_k)$, yielding
\begin{equation}\label{eq-a1_fourier}
(a_1, \varphi_k) = \frac{(u(\cdot, T), \varphi_k)}{TE_{\alpha, 2}(-\lambda_k T^{\alpha})}.
\end{equation}
We must exercise particular caution, as the denominator $tE_{\alpha,2}(-\lambda_n t^\alpha)$ in the aforementioned expression could potentially be zero. To ensure that the denominator remains non-zero, we can refer to Lemma \ref{coro-<ml<}, which indicates that we simply need to choose $t=T$ so that $T^\alpha \not\in \cup_{k=1}^\infty \{ \frac{t_1}{\lambda_k}, \cdots, \frac{t_N}{\lambda_k}\}$. Consequently, the representation
\[
a_1 = \sum_{n=1}^{\infty} \frac{(u(\cdot,T), \varphi_n)}{T E_{\alpha, 2}(-\lambda_n T^\alpha)} \varphi_n
\]
is meaningful. Now taking $ \| \cdot \|_{D((-\Delta)^{-\gamma})}^{2} $ on both sides, one can deduce that
\[
\|a_1\|_{D((-\Delta)^{-\gamma})}^2 = \sum_{n=1}^{\infty} \frac{(u(\cdot,T), \varphi_n)^2}{[T E_{\alpha, 2}(-\lambda_n T^\alpha)]^2} \lambda_n^{-2\gamma}.
\]
By Lemma \ref{lem-ml2-lower}, we have the following estimate
\begin{align*}
 \|a_1\|_{D((-\Delta)^{-\gamma})}^2 & \leq C_T \sum_{n=1}^{\infty} \frac{1}{T^2} \frac{(u(\cdot,T), \varphi_n)^2}{(\frac{1}{1 + \lambda_n T^\alpha})^2} \lambda_n^{-2\gamma} \\
  & \leq C_T \sum_{n=1}^{\infty} T^{2\alpha - 2} \lambda_n^{2-2\gamma} (u(\cdot,T), \varphi_n)^2
\leq C_T T^{2 \alpha - 2} \|(-\Delta)^{1-\gamma} u(\cdot,T)\|_{L^2(\Omega)}^2.
\end{align*}
Thus, we have
\[
\|a_1\|_{D((-\Delta)^{-\gamma})} \leq C_T \|u(\cdot,T)\|_{D((-\Delta)^{1-\gamma})}.
\]

Moreover, in the case of $\alpha\in(1,\frac43]$, the denominator in equation \eqref{eq-a1_fourier} is nonzero, guaranteed by Corollary \ref{coro-<ml<} with $\alpha\in(1,\frac43]$. Consequently, similar to the above argument, we conclude that
\begin{align*}
\|a_1\|_{D((-\Delta)^{-\gamma})}^2
&=  \sum_{n=1}^{\infty} \frac{\left|(u(\cdot,T), \varphi_n)\right|^2} {\left|TE_{\alpha, 2}(-\lambda_n T^{\alpha})\right|^2} \lambda_n^{-2\gamma}
\\
& \leq C \sum_{n=1}^{\infty} \lambda_n^{2-2\gamma} T^{2\alpha-2} |(u(\cdot,T), \varphi_n)|^2
\leq C T^{2 \alpha - 2} \|(-\Delta)^{1-\gamma} u(\cdot,T)\|_{L^2(\Omega)}^2.
\end{align*}
Collecting all the above estimates, we complete the proof of the theorem.
\end{proof}

The above stability result for our inverse problem requires the $H^2(\Omega)$-norm as the upper bound in the case of $\gamma=0$. A slight modification of the above estimates can relax the norm of the right-hand side of the stability estimate by an a priori assumption on the initial value.
%The stability result above for our inverse problem requires the $H^2(\Omega)$-norm as the upper bound. A slight modification of the above estimates, we can relax the norm of the right-hand side of the stability estimate by an argument similar to Theorem \ref{thm-condi-sub}.
\begin{theorem}\label{thm-condi-sup}
Letting $\alpha\in(1,2)$, $\beta,\gamma\in [0,1]$ be such that $\gamma+\beta\neq0$, we suppose that the pair $(u,a_1)$ in the space $L^2 \left(0,T; H_0^1(\Omega) \cap H^2(\Omega) \right) \times D((-\Delta)^{\beta})$ is a solution of our backward problem, which corresponds to the measurement data $u(\cdot,T)$. If $\|(-\Delta)^{\beta} a_1\|_{L^2(\Omega)} \le M$ for some positive constant $M$, then the estimate
$$
\|a_1\|_{D((-\Delta)^{-\gamma})}\le  C_T M^{\frac{1-\gamma}{1+\beta}} \left\| u(\cdot, T)\right\|_{L^2(\Omega)}^{\frac{\beta+\gamma}{\beta+1}}.
$$
is valid provided that one of the following conditions hold:
\begin{enumerate}[label = \roman*.]
    \item $\alpha\in(1,\frac43]$ and $T>0$;
    \item $\alpha\in(\frac43,2)$ and $T^\alpha \not\in \cup_{k=1}^\infty \{ \frac{t_1}{\lambda_k}, \cdots, \frac{t_N}{\lambda_k}\}$.
\end{enumerate}
Here $C_T>0$ is a constant that is independent of $a_1$ and $u$, but may depend on $\alpha,\beta,\gamma,T$, and $\Omega$.
\end{theorem}
\begin{proof}
Using the Cauchy-Schwarz inequality
$$
\sum_{k=1}^\infty \xi_k \zeta_k \le  \left(\sum_{k=1}^\infty \xi_k^p \right)^{1/p} \left(\sum_{k=1}^\infty  \zeta_k^q\right)^{1/q}
$$
with $p=\frac{1+\beta}{1-\gamma}$ and $q=\frac{\beta+1}{\beta+\gamma}$, here we should understand $p=+\infty$ if $\gamma=1$, we see that
$$
\begin{aligned}
\|a_1\|_{D((-\Delta)^{-\gamma})}^2
&= \sum_{n=1}^\infty (a_1,\varphi_n )_{L^2(\Omega)}^2 \lambda_n^{-2\gamma} \\
&= \sum_{n=1}^\infty \left( \lambda_n^{\frac{2\beta(1-\gamma)}{1+\beta}} |(a_1,\varphi_n )|^{\frac{2(1-\gamma)}{1+\beta}} \right) \left( \lambda_n^{-2\gamma-\frac{2\beta(1-\gamma)}{1+\beta}} |(a_1,\varphi_n )|^{2 - \frac{2(1-\gamma)}{1+\beta}} \right)\\
&\le \left(\sum_{n=1}^\infty  \lambda_n^{2\beta} |(a_1,\varphi_n )_{L^2(\Omega)}|^{2} \right)^{\frac{1-\gamma}{1+\beta}} \left(\sum_{n=1}^\infty  \lambda_n^{-2} |(a_1,\varphi_n )_{L^2(\Omega)}|^{2} \right)^{\frac{\beta+\gamma}{\beta+1}} ,
\end{aligned}
$$
which combined with the assumption $\|(-\Delta)^{\beta} a_1\|_{L^2(\Omega)} \le M$ further implies that
$$
\begin{aligned}
\|a_1\|_{D((-\Delta)^{-\gamma})}^2
\le M^{\frac{2(1-\gamma)}{1+\beta}} \left(\sum_{n=1}^\infty  \lambda_n^{-2} |(a_1,\varphi_n )_{L^2(\Omega)}|^{2} \right)^{\frac{\beta+\gamma}{\beta+1}}.
\end{aligned}
$$
Moreover, from the estimate in Theorem \ref{thm-isp} with $\gamma=1$, we see that
$$
\begin{aligned}
\|a_1\|_{D((-\Delta)^{-\gamma})}^2
\le C_T M^{\frac{2(1-\gamma)}{1+\beta}} \left\| u(\cdot, T)\right\|_{L^2(\Omega)}^{\frac{2(\beta+\gamma)}{\beta+1}},
\end{aligned}
$$
which completes the proof of the theorem.
\end{proof}

\section{Scattered point measurement-based regularization}\label{sec-num}
\subsection{Settings}
For stating the Tikhonov regularization method based on scattered point measurement, we collect a set of  scattered points $ \{ x_i \}_{i = 1}^n$ which are such that $ x_i \neq x_j $ for $ i \neq j $ and are quasi-uniformly distributed in $ \Omega $, that is, there exists a positive constant $ B $ such that $ d_{\max} \le B d_{\min} $ , where $ d_{\max}>0$ and $ d_{\min}>0$ are defined by
$$
d_{\max} = \sup_{x \in \Omega} \inf_{1 \leq i \leq n} | x - x_i |,\quad
d_{\min} = \inf_{1 \leq i \neq j \leq n} | x_i - x_j | .
$$
Furthermore, for any $u,v\in C(\overline\Omega)$ and $y\in\mathbb R^n$, we define
$$
(y,v)_n := \frac1n \sum_{i=1}^n y_i v(x_i),\quad (u,v)_n:=\frac1n \sum_{i=1}^n u(x_i)v(x_i),
$$
and the discrete semi-norm
$$
\|u\|_n := \left(\sum_{i=1}^n \frac{u^2(x_i)}{n} \right)^{\frac12},\quad u\in C(\overline \Omega).
$$
We have the inequalities connecting discrete semi-norm and the Sobolev norms.
\begin{lemma}\cite[Theorems 3.3 and 3.4]{utreras1988convergence}
\label{lem-u-un}
There exists a constant  $C>0$  such that for all  $  u \in H^k(\Omega)$ with $k>\frac{d}2$ , the following estimates are valid:
\begin{equation}\label{esti-u<un}
\begin{aligned}
\| u \|^2_{L^2(\Omega)} \leq& C \left( \| u \|^2_n + n^{-\frac{2k}{d}} \| u \|^2_{H^k(\Omega)} \right),\\
\| u \|^2_n \leq& C \left( \| u \|^2_{L^2(\Omega)} + n^{-\frac{2k}{d}} \| u \|^2_{H^k(\Omega)} \right).
\end{aligned}
\end{equation}
\end{lemma}
Now let $S$ be a linear operator from $L^2(\Omega)$ to $H^2(\Omega)$ which is defined as follows:
$$
Sa_1(\cdot) = u(\cdot,T).
$$
In view of the stability results in Theorem \ref{thm-isp}, it follows that the forward operator $S$ is bounded and one-to-one from $L^2(\Omega)$ to $H^2(\Omega)$. Moreover, let $a^*\in L^2(\Omega) $ be the unknown initial value of the problem \eqref{eq-gov}. We assume that the measurement data contains noise and is presented in the following form:
\begin{equation}\label{ob-noise0}
m_i = (Sa^*)(x_i) + e_i, \quad i = 1,2,...,n,
\end{equation}
where $\{e_i\}_{i=1}^n$ denote a sequence of random variables that are independent and identically distributed over the probability space $(X, \mathscr{F}, \mathbb{P})$. The expectation is such that $\mathbb{E}[e_i] = 0$, and the variances are bounded by $\sigma^2$, that is, $\mathbb{E}[e_i^2] \leq \sigma^2$.

We denote $\mathbf{x}:=(x_1,x_2,\cdots,x_n)$, $\mathbf{m} = (m_1, m_2, ..., m_n)^T$, and $\mathbf{e}:=(e_1,e_2,\cdots,e_n)$. Then the above term \eqref{ob-noise0} can be rephrased as the vector form:
\begin{equation}\label{ob-noise}
\mathbf{m} = (Sa^*)(\mathbf{x}) + \mathbf{e}.
\end{equation}

We seek a numerical solution, denoted as $ a_n^*$, for the unknown initial value $a^*$, utilizing the Tikhonov regularization form:
\begin{equation}\label{Tik-vec}
\arg\min_{a \in X} \| (Sa)(\mathbf{x}) - \mathbf{m}\|_n^2 + \rho_n \| a \|^2_{X},
\end{equation}
where $X:=D((-\Delta)^\gamma)$ with $\gamma\in \mathbb R$ and $ \rho_n > 0 $ is called a regularization parameter.

\subsection{Useful lemmas}
Recalling the set of positive roots of the Mittag-Leffler function $E_{\alpha,2}(-t)$: $\{t_k\}_{k=1}^N$, and Corollary \ref{coro-<ml<}, we can show the asymptotic behavior of the singular value of the forward operator $S$. More precisely, we have
\begin{lemma}\label{lem-eigen-H1}
Assuming one of the following conditions:
\begin{enumerate}[label = \roman*.]
    \item $\alpha\in (1,\frac43]$ and $T>0$;
    \item $\alpha\in(\frac43,2)$ and $T^\alpha \not\in \cup_{k=1}^\infty \{ \frac{t_1}{\lambda_k}, \cdots, \frac{t_N}{\lambda_k}\}$.
\end{enumerate}
is valid. Letting $X=D((-\Delta )^\gamma)$ with $\gamma>0$, then the eigenvalues  $  0 < \mu_1 \leq \mu_2 \leq \cdots  $  of the eigenvalue problem
\begin{equation}\label{eigen-S*S}
\mu (S\psi, Sv)_{L^2(\Omega)}=(\psi, v)_{X}=
((-\Delta)^{\gamma}\psi, (-\Delta)^{\gamma}v)_{L^2(\Omega)} ,\quad \forall v \in X
\end{equation}
satisfy that
$$
\mu_k \geq  C k^{\frac{4(1+\gamma)}d},\quad k = 1, 2, \ldots,
$$ for some constant  $  C >0 $  depending only on the operator $S$.
\end{lemma}
\begin{proof}
Firstly, we give a lower estimate for the eigenvalue $\eta$ of the following eigenvalue problem
\begin{equation}\label{eigen-S}
\psi = \eta S\psi,\quad \psi \in X.
\end{equation}
Recalling the Dirichlet eigensystem $\{\lambda_k, \varphi_k\}_{k=1}^\infty$ of the operator $-\Delta$, we first set $a = \varphi_k$. Then the solution $u$ at $t=T$ to the problem \eqref{eq-gov} can be represented by
$$
u(\cdot,T) = TE_{\alpha,2}(-\lambda_k T^\alpha)\varphi_k.
$$
Moreover, from assumptions in this lemma, we can see that $TE_{\alpha,2}(-\lambda_k T^\alpha)\neq0$ for $k\in\mathbb N^+$. Consequently, according to the definition of the forward operator $S$, we have
\begin{equation}\label{eq-eta_k}
S(\varphi_k) = TE_{\alpha,2}(-\lambda_k T^\alpha)\varphi_k=:\eta_k^{-1} \varphi_k,
\end{equation}
which means that $\{\eta_k, \varphi_k\}_{k=1}^\infty$ is an eigensystem of the problem \eqref{eigen-S}. Now by letting $\psi=v=\varphi_k$ in \eqref{eigen-S*S}, we see that
$$
\left((-\Delta)^{\gamma}\varphi_k, (-\Delta)^{\gamma} \varphi_k  \right) = \mu (S\varphi_k, S\varphi_k),
$$
from which by further noting  \eqref{eq-eta_k} and the facts that $(-\Delta)^{\gamma} \varphi_k = \lambda_k^{\gamma} \varphi_k$, we see that
$$
\mu_k = \lambda_k^{2\gamma} \eta_k^2.
$$
Moreover, under assumptions in this lemma, we conclude from Corollary \ref{coro-<ml<} that the eigenvalue $\eta_k$ of the problem \eqref{eigen-S} admits the following lower bound:
$$
\eta_k = \frac1{|TE_{\alpha,2}(-\lambda_k T^\alpha)|}\ge C\lambda_k,
$$
which further combined with the asymptotic estimate $\lambda_k \sim k^{\frac{2}{d}}$ as $k\to\infty$ implies that $\eta_k \ge Ck^{\frac2d}, \forall k \in \mathbb N^+.
$ Finally, we see that $\mu_k = \lambda_k^{2\gamma} \eta_k^2 \ge Ck^{\frac{4(1+\gamma)}d}$ for any integer $k\ge1$. This completes the proof of the lemma.
\end{proof}

The Lemma \ref{lem-Vn} shows that the solution is attained in a
finite dimensional subspace $V_n$, which is necessary in stochastic convergence analysis of Lemma \ref{lem-Ean-Ea*}.
\begin{lemma}\label{lem-Vn}
Given $\mathbf{m}\in \mathbb R^n$ and let $\mathbf{x}=\{x_i\}_{i=1}^n$ be scattered points which is quasi-uniformly distributed in $\Omega$, then there exists an $n$-dimensional subspace $V_n$ of the Hilbert space $X:=D((-\Delta)^\gamma)$, $\gamma\ge0$ such that
$$
\min_{a\in X, Sa(\mathbf{x})=\mathbf{m}} \|a\|_{X}^2 = \min_{a\in V_n, Sa(\mathbf{x})=\mathbf{m}} \|a\|_{X}^2.
$$
\end{lemma}
\begin{proof}
Firstly, we define a subset $V\subset X$ which collects all the functions $a\in X$ such that $(Sa)(\mathbf x)=0$. Since the operator $S$ is linear, it is not difficult to check that $V$ is a linear subspace of $X$. Then we define the projection operator $P_V:X\to V$ by
$$
 (P_V[a], v)_{X} = (a,v)_{X},\quad \forall v\in V.
$$
Now in view of the fact that the forward operator $S:L^2(\Omega) \to H^2(\Omega)$ is one-to-one, we can choose a sequence of functions $\{\phi_i\}_{i=1}^n\subset X$ such that $(S\phi_i)(x_j)= \delta_{ij}$, $i,j=1,\cdots,n$. Taking $\psi_i := \phi_i - P_V[\phi_i]$, $i=1,\cdots,n$, we can easily check that $(S\psi_i)(x_j)= \delta_{ij}$ also holds for any $i,j=1,\cdots,n$. We further define the linear space $V_n := {\rm span} \{\psi_1,\cdots,\psi_n\}$, and for any $a\in X$, we define the interpolation operator $I: X\to V_n$ as follows:
$$
Ia= \sum_{i=1}^n (Sa)(x_i)\psi_i,\quad a\in X.
$$
Combining with the definitions of the subspaces $V_n$ and $V$ of $X$, we immediately see that $Ia\in V_n$ and $a-Ia\in V$, hence that $(v,\phi_i - P_V[\phi_i])_{X} = 0$ for all $v\in V$, which further implies that
$$
\begin{aligned}
(a-Ia, Ia)_{X}
&= \left(a - Ia, \sum_{i=1}^n (Sa)(x_i)(\phi - P_V[\phi_i])\right)_{X} \\
&= \sum_{i=1}^n  (Sa)(x_i)(a - Ia,\phi_i - P_V[\phi_i])_{X}=0,
\end{aligned}
$$
that is, $(Ia, Ia)_{X}= (a, Ia)_{X}$. Finally, by employing the Cauchy-Schwarz inequality, we arrive at the inequality $(Ia, Ia)_{X} \le  (a, a)_{X}$, hence we have
$$
\min_{a\in V_n, Sa(\mathbf{x})=\mathbf{m}} \|a\|^2_{X} =  \min_{a\in X,Sa(\mathbf{x})=\mathbf{m}} \|a \|^2_{X}.
$$
This completes the proof of the lemma.
\end{proof}
\begin{lemma}\label{lem-eigen-Sn}
Under the same assumptions in Lemmas \ref{lem-eigen-H1} and \ref{lem-Vn}, then the space $V_n$ becomes an $n$-dimensional Hilbert space with the inner product $(S\cdot,S\cdot)_n$. Moreover, the eigenvalue problem
\begin{equation}\label{eigen-Vn}
( \psi, v )_{X} = \mu^{(n)} ( S \psi, S v )_n, \quad \forall v \in V_n
\end{equation}
has $ n $ eigenvalues $0< \mu_1^{(n)} \leq \mu_2^{(n)}  \leq \cdots \leq \mu_n^{(n)} <\infty$, and all the eigenfunctions form an orthogonal basis of $ V_n $ with respect to the inner product $ ( S \cdot, S\cdot)_n $. Moreover, there exists a constant $ C > 0 $ independent of $ k ,n$ such that
$$
\mu_k^{(n)} \geq C k^{\frac{4(1+\gamma)}{d}},\quad k = 1, 2, \ldots, n.
$$
\end{lemma}
\begin{proof}
Recalling the linear subspace $V_n = \text{span} \{ \psi_i \}_{i=1}^n $ of $X$ as defined in the proof of Lemma \ref{lem-Vn}, we see that $(S\psi_i)(x_j) = \delta_{ij} $. Moreover, it is not difficult to see that $\{\psi_i\}_{i=1}^n$ is a basis for the space $V_n$ and is such that $(S\psi_i,S\psi_j)_n = \delta_{ij}$, $i,j=1,\cdots,n$.

We now assert that the space $V_n$ equipped with the inner product $(S\psi,S\phi)_n$, $\psi,\phi\in V_n$, becomes a Hilbert space of $n$-dimension. Consequently, the eigenvalue problem \eqref{eigen-Vn} is equivalent to a matrix eigenvalue problem
$$
\mathbf{A} \Psi = \mu^{(n)}  \mathbf{B} \Psi  \text{ for }  \Psi \in \mathbb{R}^n ,
$$
where $ \mathbf{A}, \mathbf{B} \in \mathbb{R}^{n \times n} $ are two symmetric positive definite matrices. Thus, the eigenvalue problem \eqref{eigen-Vn} has $n$ eigenvalues $ 0<\mu_1^{(n)} \leq \mu_2^{(n)} \leq \cdots \leq \mu^{(n)} _n <\infty$ and all eigenfunctions form an orthogonal basis of $ V_n $ with respect to the inner product $(S \cdot,S \cdot)_n$.

Now we give a lower bound of $ \mu_k^{(n)}$ for $k=1,2,\cdots,n$. By the min-max principle of the Rayleigh quotient for the eigenvalues and \eqref{esti-u<un} in Lemma \ref{lem-u-un}, we see that
$$
\begin{aligned}
\mu_k^{(n)} &= \min_{\substack{\text{dim}W=k,\\ W\subset V_n}} \max_{a \in W} \frac{(a,a)_{X}}{(Sa,Sa)_n}
\\
&\geq C \min_{\substack{\text{dim}W=k,\\ W\subset V_n}} \max_{a \in W} \frac{(a,a)_{X}}{(Sa,Sa)_{L^2(\Omega)} + n^{- 4(\gamma+1)/d} (Sa,Sa)_{H^{2(\gamma+1)}(\Omega)}}
\\
&\geq C \min_{\substack{\text{dim}W=k,\\ W\subset V_n}} \max_{a \in W} \frac{(a,a)_{X}}{(Sa,Sa)_{L^2(\Omega)} + n^{- 4(\gamma+1)/d} (Sa,Sa)_{D((-\Delta)^{\gamma+1})}}.
\end{aligned}
$$
Here the last inequality is due to the fact $\|\phi\|_{H^{2(\gamma+1)}(\Omega)} \le C\|\phi\|_{D((-\Delta)^{\gamma+1})}$ if $\gamma\ge0$. Moreover, by noting the estimates $(Sa,Sa)_{D((-\Delta)^{\gamma+1})}\leq C(a,a)_{D((-\Delta)^{\gamma})}$, we see that
$$
\begin{aligned}
\mu_k^{(n)}
&\geq C \min_{\substack{\text{dim}W=k,\\ W\subset X}} \max_{a \in W} \frac{(a,a)_{X}}{(Sa,Sa)_{L^2(\Omega)} + n^{- 4(\gamma+1)/d} (a,a)_{D((-\Delta)^\gamma)}}
\\
&= C \min_{\substack{\text{dim}W=k,\\ W\subset X}} \max_{a \in W} \frac{(a,a)_{X}}{(Sa,Sa)_{L^2(\Omega)} + n^{- 4(\gamma+1)/d} (a,a)_X},\quad \gamma\ge0,
\end{aligned}
$$
which again combined with the Rayleigh quotient for the eigenvalues implies that
$$
\begin{aligned}
\mu_k^{(n)}
\geq C  \frac{1}{\mu_k^{-1} + n^{- 4(\gamma+1)/d}}.
\end{aligned}
$$
Finally, by using the results from Lemma \ref{lem-eigen-H1}, we complete the proof of the lemma.
\end{proof}

\subsection{Stochastic convergence analysis}
In this part, we will give the main result for the stochastic convergence of the Tikhonov regularization scheme \eqref{Tik-vec}. To this end, we first give a useful lemma.
\begin{lemma}\label{lem-Ean-Ea*}
Let $a^*\in X:=D((-\Delta)^\gamma)$ with $\gamma\ge0$ and let $a_n^* \in X$ be the unique solution of our Tikhonov regularization form \eqref{Tik-vec}. Then there exist constants $ \rho_0 > 0 $ and $ C > 0 $ such that the following estimates
\[
\mathbb{E} \bigl[ \| S a_n^* - S a^* \|^2_n \bigr] \leq C \left( \rho_n \| a^* \|^2_X +  \frac{\sigma^2}{n \rho_n^{\frac{d}4/(1+\gamma)}} \right)
\]
and
\[
\mathbb{E} \bigl[ \| a_n^* - a^* \|^2_{X} \bigr] \leq C \left( \| a^* \|^2_X + \frac{\sigma^2}{n \rho_n^{1 + \frac{d}4/(1+\gamma)}}\right).
\]
are valid for any $0< \rho_n \leq \rho_0 $. Here, the constant $C$ is independent of $n,a^*,a_n^*$.
\end{lemma}
\begin{proof}
In view of the argument in Theorem 2.3 in \cite{chen2022stochastic}, for problem \eqref{Tik-vec}, we can easily find that the minimizer $a_n^* \in X$ is exists and is unique and admits the variational form
\begin{equation}\label{eq-varia}
\rho_n (a_n^*, v)_X + (Sa_n^*, Sv)_n = (\mathbf{m}, Sv)_n, \quad \forall v \in X.
\end{equation}
For any $v \in X$, we define the energy norm $\| \cdot \|_{\rho_n}$ as follows:
$$
\| v \|_{\rho_n}^2 := \rho_n \|v\|_X^2 + \| Sv \|_n^2.
$$
By setting $v = a_n^* - a^*$ in equation \eqref{eq-varia}, and considering \eqref{ob-noise}, we derive the following:
\begin{equation}\label{esti-an-a*}
\| a_n^* - a^* \|_{\rho_n}^2 \leq C\rho_n \| a^* \|_X^2 + C\sup_{v \in X} \frac{|(e, Sv)_n|^2}{\| v \|_{\rho_n}^2}.
\end{equation}
We still need to determine the supremum term in equation \eqref{esti-an-a*}. With the aid of Lemma \ref{lem-Vn}, it is possible to rewrite it in an equivalent form as follows:
$$
\begin{aligned}
  \sup_{v \in X} \frac{(e, Sv)_n^2}{\| v \|_{\rho_n}^2}
  \leq& \sup_{v \in X} \frac{(e, Sv)_n^2}{\| Sv \|_n^2 + \rho_n \min_{u \in X, Su(\mathbf{x}) = Sv(\mathbf{x})} \|u\|_X^2 }
  \\
  =& \sup_{v \in X} \frac{(e, Sv)_n^2}{\| Sv \|_n^2 + \rho_n \min_{u \in V_n, Su(\mathbf{x}) = Sv(\mathbf{x})} \|u\|_X^2 }
=  \sup_{v \in V_n} \frac{(e, Sv)_n^2}{\rho_n  \|v\|_X^2 + \| Sv \|_n^2}.
\end{aligned}
$$
Now in view of Lemma \ref{lem-eigen-Sn}, we can choose the eigensystem $\{\mu_k^{(n)}, \psi_k\}_{k=1}^n$ of the problem  \eqref{eigen-Vn}. Moreover, the eigenvalue and the corresponding eigenfunctions are such that $0<\mu_1^{(n)} \leq \mu_2^{(n)} \leq \cdots \leq \mu_n^{(n)}<\infty$ and $\{\psi_k\}_{k=1}^n$ form an orthonormal basis of $V_n\subset X$ under the inner product $(S \cdot, S \cdot)_n$, that is,
\begin{equation}\label{eq-SkSl}
(S\psi_k, S\psi_l)_n = \delta_{kl}, \quad k, l = 1, 2, \ldots, n.
\end{equation}
Consequently, by noting $\psi_k$ is the eigenfunction of the problem \eqref{eigen-Vn}, we see that
\begin{equation}
    \label{eq-pkpl}
(\psi_k, \psi_l)_X=\mu_k^{(n)} (S\psi_k,S\psi_l)_n = \mu_k^{(n)} \delta_{kl}, \quad k, l = 1, 2, \ldots, n.
\end{equation}
Now we conclude from \eqref{eq-SkSl} that $v\in V_n$ can be further expressed as
\begin{equation}\label{eq-v}
v(x) = \sum_{k=1}^n (Sv, S\psi_k)_n\, \psi_k(x),  \quad k = 1, 2, \ldots, n.
\end{equation}
By denoting $v_k := (Sv, S\psi_k)_n$, and noting the equation \eqref{eq-pkpl}, we see that
\begin{equation}\label{eq-|v|}
\|v\|_{\rho_n}^2 = \sum_{k=1}^n (\rho_n \mu_k^{(n)} + 1) v_k^2.
\end{equation}
From the representation \eqref{eq-v}, and noting the definition of $\psi_k$, we conclude from the Cauchy-Schwarz inequality that
\begin{align*}
(e, Sv)_n^2 &= \frac{1}{n^2}  \left[\sum_{i=1}^n e_i \left( \sum_{k=1}^n v_k S\psi_k(x_i) \right) \right]^2 \\
&= \frac{1}{n^2}  \left[\sum_{k=1}^n v_k \left( \sum_{i=1}^n e_i \delta_{k,i} \right) \right]^2
= \frac{1}{n^2}  \left[\sum_{k=1}^n v_k  e_k\right]^2 \\
&\leq \frac{1}{n^2} \sum_{k=1}^n (1 + \rho_n \mu_k^{(n)}) v_k^2 \cdot \sum_{k=1}^n (1 + \rho_n \mu_k^{(n)})^{-1}  e_k^2,
\end{align*}
from which we further use the equation  \eqref{eq-|v|} to derive that
\[
\begin{aligned}
\mathbb{E} \left[ \sup_{v \in V_n} \frac{(e, Sv)_n^2}{\|v\|_{\rho_n}^2} \right] \leq&\frac{1}{n^2} \sum_{k=1}^n (1 + \rho_n \mu_k^{(n)})^{-1} \mathbb{E} \left[\sum_{k=1}^n e_k^2 \right]
\leq \frac{\sigma^2} n \sum_{k=1}^n (1 + \rho_n \mu_k^{(n)})^{-1},
\end{aligned}
\]
where in the derivation of the last estimate we used the fact that $\mathbb{E}[e_i e_j] = \delta_{ij}$. Moreover, by Lemma \ref{lem-eigen-Sn}, we readily derive
\[
\begin{aligned}
\mathbb{E} \left[ \sup_{v \in X} \frac{(e, Sv)_n^2}{\|v\|_{\rho_n}^2} \right]
\leq \frac{C \sigma^2} n \sum_{k=1}^n (1 + \rho_n k^{\frac{4(1+\gamma)}{d}})^{-1}
\leq \frac{C \sigma^2} n \int_1^\infty (1 + \rho_n t^{\frac{4(1+\gamma)}{d}})^{-1} \, dt.
\end{aligned}
\]
Setting $\beta:={\frac{4(1+\gamma)}{d}}$ and noting that $\gamma\ge0$, then it is easy to see that
\[
\int_1^\infty (1 + \rho_n t^{\beta})^{-1} \, dt = \rho_n^{-1/\beta} \int_{\rho_n^{1/\beta}}^\infty (1 + s^\beta)^{-1} \, ds \leq C \rho_n^{-1/\beta}.
\]
This completes the proof of the lemma by using \eqref{esti-an-a*}.
\end{proof}
As an application, we can derive a convergence estimate for the regularized solution $ a_n^* $ to approximate the unknown initial value $ a^* $.
\begin{theorem}\label{thm-E-condi}
Under the same assumptions in Lemma \ref{lem-eigen-H1}, we further suppose that the unknown initial value $a^*\in X:=D((-\Delta)^\beta)$ with $\beta>0$, then there exists a constant $C=C(\beta,\rho_n,T,\Omega)$ such that
\[
\begin{aligned}
\mathbb{E} \bigl[ \| a_n^* - a^* \|^{2+\frac{2}{\beta}}_{L^2(\Omega)} \bigr]
\le  C\left( \| a^* \|^2_{X} + \frac{\sigma^2}{n \rho_n^{1 + \frac{d}4/(1+\beta)}}\right)^{\frac2\beta} \left( \frac1{n^{4(1+\beta)/d}} +\rho_n \right).
\end{aligned}
\]
\end{theorem}
\begin{proof}
In view of the estimate \eqref{esti-u<un}, we see that
\[
 \| S a_n^* - S a^* \|^2_n  + n^{-\frac4d}  \| S a_n^* - S a^* \|^2_{H^2(\Omega)}  \geq C  \| S a_n^* - S a^* \|^2_{L^2(\Omega)}.
\]
Moreover, from the second estimate in Lemma \ref{lem-Ean-Ea*} with $X=D((-\Delta)^{\beta})$, combined with Theorem \ref{thm-condi-sup} with $\gamma=0$ further implies that
\[
 \| S a_n^* - S a^* \|^2_n  + n^{-\frac4d} \| S a_n^* - S a^* \|^2_{H^2(\Omega)}  \geq C  M_n^{-\frac{2}{\beta}}\| a_n^* - a^* \|^{2+\frac{2}{\beta}}_{L^2(\Omega)} .
\]
Here $M_n:= \|(-\Delta)^\beta (a_n^* - a^*)\|_{L^2(\Omega)}$. On the other hand, Lemma \ref{lem-forward} implies that
\[
 \| S a_n^* - S a^* \|^2_{H^2(\Omega)} \leq C \|a_n^* - a^* \|^2_{L^2(\Omega)} .
\]
Collecting all the above estimates and using the first inequality in Lemma \ref{lem-Ean-Ea*} with $\gamma=\beta$, we see that
\[
\mathbb{E} \bigl[ \| a_n^* - a^* \|^{2+\frac{2}{\beta}}_{L^2(\Omega)} \bigr]\le C\mathbb{E} \bigl[M_n^{\frac2\beta} \bigr] \left[n^{-\frac4d}\mathbb{E} \bigl[\|a_n^* - a^* \|^2_{L^2(\Omega)} \bigr] +  \rho_n \| a^* \|^2_{X} +  \frac{\sigma^2}{n \rho_n^{\frac{d}{4}/(1+\beta)}} \right].
\]
Moreover, for $\varepsilon>0$, by the use of the Young inequality $|\xi\zeta| \le \frac{C}{\varepsilon} |\xi|^p + \varepsilon |\zeta|^q$ with $p=\beta+1$ and $q=\frac{\beta+1}{\beta}$, we obtain
\[
 n^{-\frac4d}\|a_n - a^* \|^2_{L^2(\Omega)}  \le \frac{C}{\varepsilon}n^{-\frac{4(1+\beta)}d} + \varepsilon \|a_n - a^* \|^{2+\frac2\beta}_{L^2(\Omega)}.
\]
Consequently, we see that
\[
\begin{aligned}
&\mathbb{E} \bigl[ \| a_n^* - a^* \|^{2+\frac{2}{\beta}}_{L^2(\Omega)} \bigr] \\
\le&  C\mathbb{E} \bigl[M_n^{\frac2\beta} \bigr] \left[ \frac{1}{\varepsilon}n^{-\frac{4(1+\beta)}d}  + \varepsilon \mathbb{E} \bigl[\|a_n^* - a^* \|^{2+\frac2\beta}_{L^2(\Omega)} \bigr] +  \rho_n \| a^* \|^2_{X} +  \frac{\sigma^2}{n \rho_n^{\frac{d}{4}/(1+\beta)}} \right].
\end{aligned}
\]
By letting $\varepsilon=\frac1{2C\mathbb{E} \bigl[M_n^{\frac2\beta} \bigr]}$, we can see that the term  $\varepsilon\mathbb{E} \bigl[\|a_n^* - a^* \|^{2+\frac2\beta}_{L^2(\Omega)} \bigr]$ on the right hand side of the above inequality can be absorbed, and then we get
\[
\begin{aligned}
\mathbb{E} \bigl[ \| a_n^* - a^* \|^{2+\frac{2}{\beta}}_{L^2(\Omega)} \bigr]
\le  C\mathbb{E} \bigl[M_n^{\frac2\beta} \bigr] \left[ 2C\mathbb{E} \bigl[M_n^{\frac2\beta} \bigr] n^{-\frac{4(1+\beta)}d}  +  \rho_n \| a^* \|^2_{X} +  \frac{\sigma^2}{n \rho_n^{\frac{d}{4}/(1+\beta)}} \right].
\end{aligned}
\]
Now by using the second estimate in Lemma \ref{lem-Ean-Ea*} with $X=D((-\Delta)^\beta)$, and noting the definition $M_n:=\|(-\Delta)^\beta (a_n^* - a^*)\|_{L^2(\Omega)}$, it follows that
$$
\mathbb{E} \bigl[M_n^{\frac2\beta} \bigr] \le C\left[ \| a^* \|^2_{X} + \frac{\sigma^2}{n \rho_n^{1 + \frac{d}4/(1+\beta)}}\right]^{\frac1\beta},
$$
from which we further know that
\[
\begin{aligned}
\mathbb{E} \bigl[ \| a_n^* - a^* \|^{2+\frac{2}{\beta}}_{L^2(\Omega)} \bigr]
\le&  C\mathbb{E} \bigl[M^{\frac2\beta} \bigr] \left[ \frac{\mathbb{E} \bigl[M^{\frac2\beta} \bigr]}{n^{4(1+\beta)/d}}  +  \rho_n \| a^* \|^2_{X} +  \frac{\sigma^2}{n \rho_n^{\frac{d}{4}/(1+\beta)}} \right]\\
\le & C\left[ \| a^* \|^2_{X} + \frac{\sigma^2}{n \rho_n^{1 + \frac{d}4/(1+\beta)}}\right]^{\frac2\beta} \frac1{n^{4(1+\beta)/d}} \\
&+ C\rho_n\left[ \| a^* \|^2_{X} + \frac{\sigma^2}{n \rho_n^{1 + \frac{d}4/(1+\beta)}}\right]^{\frac1\beta+1}.
\end{aligned}
\]
We can complete the proof of the theorem.
\end{proof}
Such results are important in practical applications because they provide a method for the selection of key parameters.
\begin{remark}\label{remark-1}
Lemma \ref{lem-Ean-Ea*} and Theorem \ref{thm-E-condi} indicate that the optimal regularization parameter has the form: $\rho_n^{\frac12 + \frac{d}8/(1+\beta)} = O(\sigma n^{-\frac12} \|a^*\|^{-1}_{X})$ for $\beta\in(0,1]$, and the optimal error will be
\[
\mathbb{E} \bigl[ \| a_n^* - a^* \|^{2+\frac2\beta}_{L^2(\Omega)} \bigr]\le C \left(\rho_n+ \frac1{n^{4(1+\beta)/d}} \right) \| a^* \|^{\frac4\beta}_{X}.
\]
\end{remark}
\begin{theorem}\label{thm-E'-condi}
Under the same assumption in Lemma \ref{lem-eigen-H1} and Theorem \ref{thm-condi-sup},   we further suppose that the unknown initial value $a^*\in X :=L^2(\Omega)$, then there exists a constant $C=C(\rho_n,T,\Omega)$ such that
\[
\begin{aligned}
\mathbb{E} \bigl[ \| a_n^* - a^* \|^4_{H^{-1}(\Omega)} \bigr]
\le   C \left( n^{-\frac4d} + \rho_n \right) \left(\| a^* \|^2_{L^2(\Omega)} +  \frac{\sigma^2}{n \rho_n^{1+\frac{d}4}}\right)^2.
\end{aligned}
\]
\end{theorem}
\begin{proof}
In view of estimate \eqref{esti-u<un}, we see that
\[
\mathbb{E} \bigl[ \| S a_n^* - S a^* \|^2_n \bigr] + n^{-\frac4d} \mathbb{E} \bigl[ \| S a_n^* - S a^* \|^2_{H^2(\Omega)} \bigr] \geq C \mathbb{E} \bigl[ \| S a_n - S a^* \|^2_{L^2(\Omega)} \bigr],
\]
which combined with Theorem \ref{thm-condi-sup} with $\beta=0,\gamma=\frac12$ further implies that
\[
\mathbb{E} \bigl[ \| S a_n^* - S a^* \|^2_n \bigr] + n^{-\frac4d}\mathbb{E} \bigl[ \| S a_n^* - S a^* \|^2_{H^2(\Omega)} \bigr] \geq C M_n^{-1} \mathbb{E} \bigl[ \| a_n^* - a^* \|^4_{H^{-1}(\Omega)} \bigr].
\]
Here $M_n\ge C\mathbb{E} \bigl[ \|a_n^* - a^*\|_{L^2(\Omega)}^2\bigr]$ will be chosen later. On the other hand, by Lemma \ref{lem-forward}, it follows that
\[
\mathbb{E} \bigl[ \| S a_n^* - S a^* \|^2_{H^2(\Omega)} \bigr] \leq C \mathbb{E} \bigl[\|a_n^* - a^* \|^2_{L^2(\Omega)} \bigr].
\]
Collecting all the above estimates and using the first estimate in Lemma \ref{lem-Ean-Ea*} with $X=L^2(\Omega)$, we arrive at the inequalities
\[
\begin{aligned}
\mathbb{E} \bigl[ \| a_n^* - a^* \|^4_{H^{-1}(\Omega)} \bigr]
\le& C M_n \left(n^{-\frac4d}M_n +  \rho_n \| a^* \|^2_{L^2(\Omega)} +  \frac{\sigma^2}{n \rho_n^{\frac{d}4}}\right)\\
\le & C n^{-\frac4d}M_n^2 +  CM_n\left(\rho_n \| a^* \|^2_{L^2(\Omega)} +  \frac{\sigma^2}{n \rho_n^{\frac{d}4}}\right).
\end{aligned}
\]
Now in view of the second inequality in Lemma \ref{lem-Ean-Ea*} with $\gamma=0$, we choose $M_n$ such that
$$
M_n=  C\bigg(\| a^* \|^2_{L^2(\Omega)} +  \frac{\sigma^2}{n \rho_n^{1 + \frac{d}4}}\bigg)
\ge C \mathbb{E} \bigl[ \|a_n^* - a^*\|_{L^2(\Omega)}^2\bigr],
$$
which implies that
\[
\begin{aligned}
\mathbb{E} \bigl[ \| a_n^* - a^* \|^4_{H^{-1}(\Omega)} \bigr]
\le  C \left( n^{-\frac4d} + \rho_n \right) \left(\| a^* \|^2_{L^2(\Omega)} +  \frac{\sigma^2}{n \rho_n^{1+\frac{d}4}}\right)^2.
\end{aligned}
\]
We complete the proof of the theorem.
\end{proof}

\begin{remark}
\label{remark-2}
Lemma \ref{lem-Ean-Ea*} and Theorem \ref{thm-E'-condi} indicate that the optimal regularization parameter has the form: $\rho_n^{\frac12 + \frac{d}8} = O(\sigma n^{-\frac12} \|a^*\|^{-1}_{L^2(\Omega)})$. The optimal error will be
\[
\mathbb{E} \bigl[ \| a_n^* - a^* \|^{4}_{H^{-1}(\Omega)} \bigr]\le C \left(\rho_n+ n^{-\frac4d} \right) \| a^* \|^4_{L^2(\Omega)}.
\]
\end{remark}
\begin{remark}
The regularization parameter choice (Remark \ref{remark-1}-\ref{remark-2} and error bounds are fully consistent with the conditional stability framework. Indeed, the error bounds in Theorems \ref{thm-E-condi} and \ref{thm-E'-condi} are specific instantiations of the conditional stability framework (Theorem \ref{thm-condi-sup}):
\begin{enumerate}
    \item Theorem \ref{thm-E-condi} corresponds to the case $\gamma = 0$ and $ X = D((-\Delta)^\beta) $ ($ \beta > 0 $), where the H\"older exponent $ \frac{\beta+\gamma}{\beta+1} $ reduces to $ \frac{\beta}{\beta+1} $, governing the $ L^2(\Omega) $-norm error bound.
    \item Theorem \ref{thm-E'-condi} corresponds to $ \beta = 0 $ and $ \gamma = \frac{1}{2} $, which matches the  $ H^{-1}(\Omega) $-norm error bound for $ L^2(\Omega) $-based regularization. This also reflects the weaker stability for less regular initial values ($\beta=0$).
\end{enumerate}
\end{remark}

\section{Numerical experiment}
In this section, we test numerical examples to reconstruct the initial value $a_1$ in \eqref{eq-gov}. All experiments are carried out in the space area $\Omega=(0,1)\times(0,1)$. The finite difference and finite element methods on uniform grids are used for temporal and spatial discretization, respectively. Following the idea in \cite{Sun2006ANM}, the diffusion-wave system is rewritten to the coupled equations based on a order reduction method. Specifically, the Caputo derivative is approximated by $L1$ method.  Let $A_h$: $V_h\rightarrow V_h$ be the discrete Laplacian operator, where $V_h$ is the finite element space. Denote $\|\phi\|_{L^2(\Omega)}:=(\phi,\phi)^{\frac12}$, $\|\phi\|_{H^{-1}(\Omega)}:=(A_h^{-1}\phi,\phi)^{\frac12}$, $\forall\phi \in V_h$.  The details of the $H^{-1}$ norm are provided in \cite{Hu2007ACM,Zou2019NMTMA}.

The scattered point measurements take the form
$$m_i = (Sa^*)(x_i) + e_i, ~~ i = 1,2,\cdots,n,$$
where $\{e_i\}_{i=1}^n$ are independent and identically distributed random variables in a probability space $(X, \mathscr{F}, \mathbb{P})$ and are such that $\mathbb E[e_i]=0$ and $\mathbb E[e_i^2] \le \sigma^2$. Measurements are evenly distributed within the region $\Omega$, forming a finite element mesh, and the number of mesh points takes $n=6241$. Furthermore, the relative $L^2$-norm and $H^{-1}$-norm error between the true initial function $a_1$ and the reconstructed one $a_1^{rec}$ are as follows.
$$err_{L^2}=\frac{\|a_1^{rec}-a_1\|_{L^2(\Omega)}}{\|a_1\|_{L^2(\Omega)}},~~err_{H^{-1}}=\frac{\|a_1^{rec}-a_1\|_{H^{-1}(\Omega)}}{\|a_1\|_{H^{-1}(\Omega)}}.$$
\begin{example}\label{ex-sigma=0.2}
We fix the noise parameter $\sigma=0.2$, the initial guess $a_1^0=0$, and the initial function is as follows:
\begin{align*}
a_1=8x^{1.01}(x-1)\sin(\pi y),~T=1.
\end{align*}
\end{example}

In Example \ref{ex-sigma=0.2}, the maximum norm of $Sa_1$ is $\|Sa_1\|_{L^\infty} \approx 0.95$, $1.30$ for $\alpha=1.2$, $1.8$. Thus, noise level $\frac{\sigma}{\|Sa_1\|_{L^\infty}}\approx24\%$, $15\%$ for $\alpha=1.2$, $1.8$.
Based on Remark \ref{remark-1}, the optimal regularization parameter $\rho = O\big((\sigma n^{-\frac12} \|a_1\|^{-1}_{H^1(\Omega)})^{3/2}\big)=O( 1.31\times 10^{-5})$ for the $H^1$ regularization. And we compare the true initial function $a_1$ in Figure \ref{figure5} and the reconstruction results. In Figure \ref{figure7}, we take the regularization parameter $\rho=3\times10^{-5}\approx 10^{-4.523}$, and the inversion result is good. To show the importance of the choices of $\rho$, we take $\rho=10^{-6}$ in Figure \ref{figure8} and the reconstruction is not that good.
The numerical results in Figures \ref{figure_add1} and \ref{figure_add2} indicate that the optimal regularization parameters are taken in the interval $(10^{-4},10^{-5})$. It agrees with our theoretical estimation.

\begin{figure}[!htb]
        \centering
	\begin{minipage}{0.49\linewidth}
		\centering
		\setlength{\abovecaptionskip}{0.28cm}
		\includegraphics[width=1\linewidth]{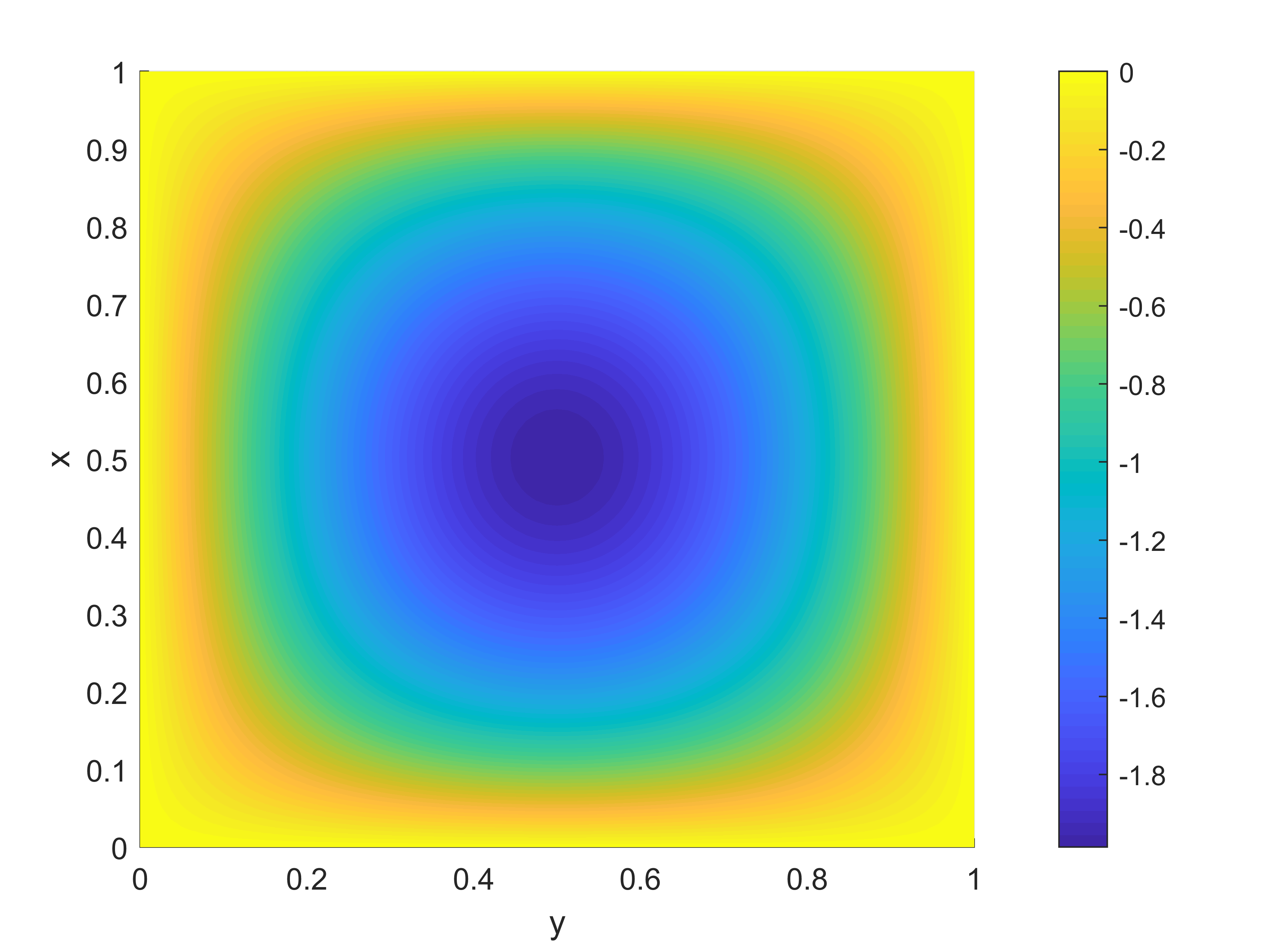}
		\caption{True $a_1$ in Ex. \ref{ex-sigma=0.2}}
		\label{figure5}
	\end{minipage}
	\begin{minipage}{0.49\linewidth}
		\centering
		\includegraphics[width=1\linewidth]{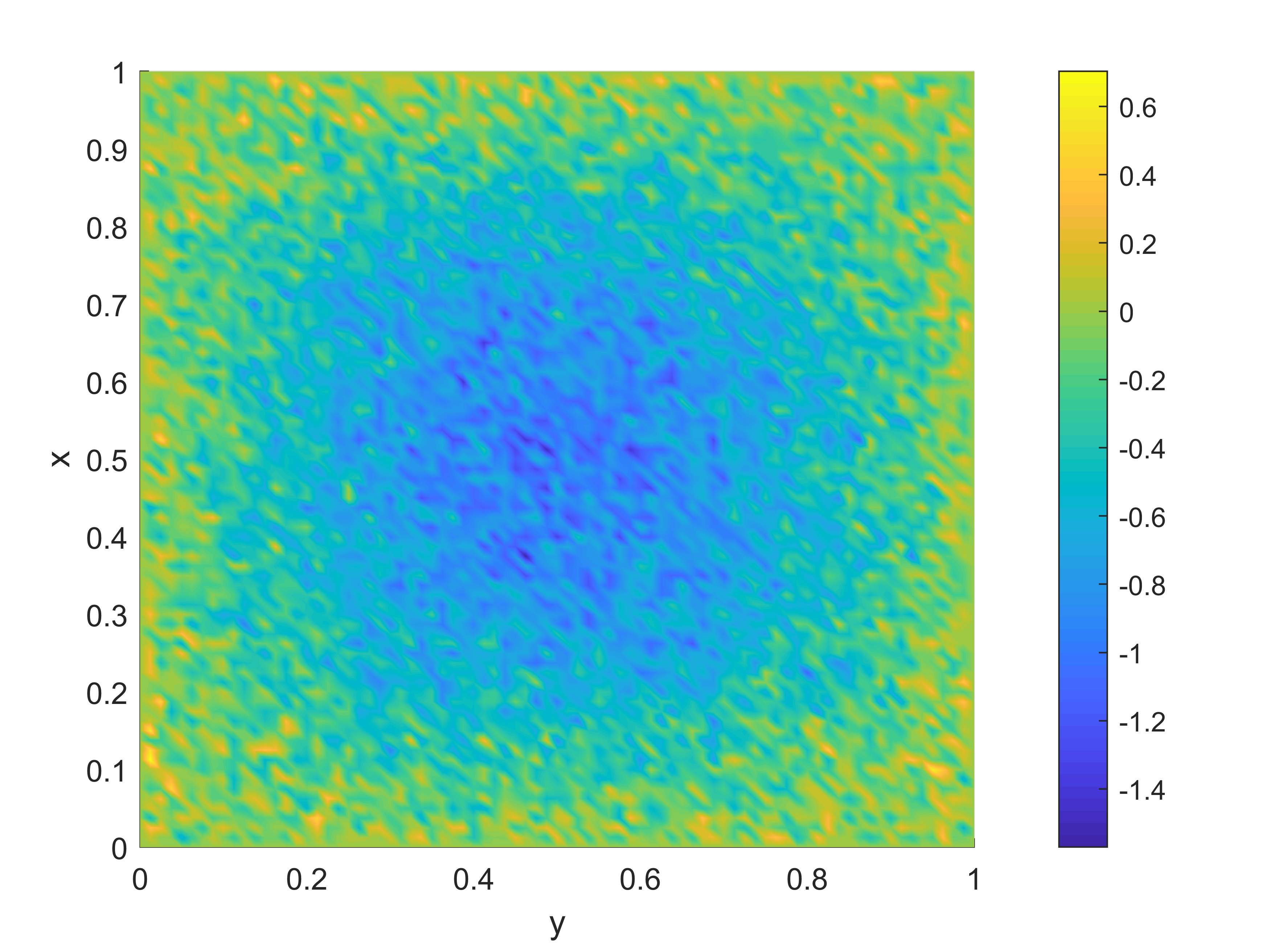}
		\caption{measurements, $\alpha=1.2$}
	\end{minipage}
        \vspace{0.3cm}
        \vfill
	\begin{minipage}{0.49\linewidth}
		\centering
		\setlength{\abovecaptionskip}{0.28cm}
		\includegraphics[width=1\linewidth]{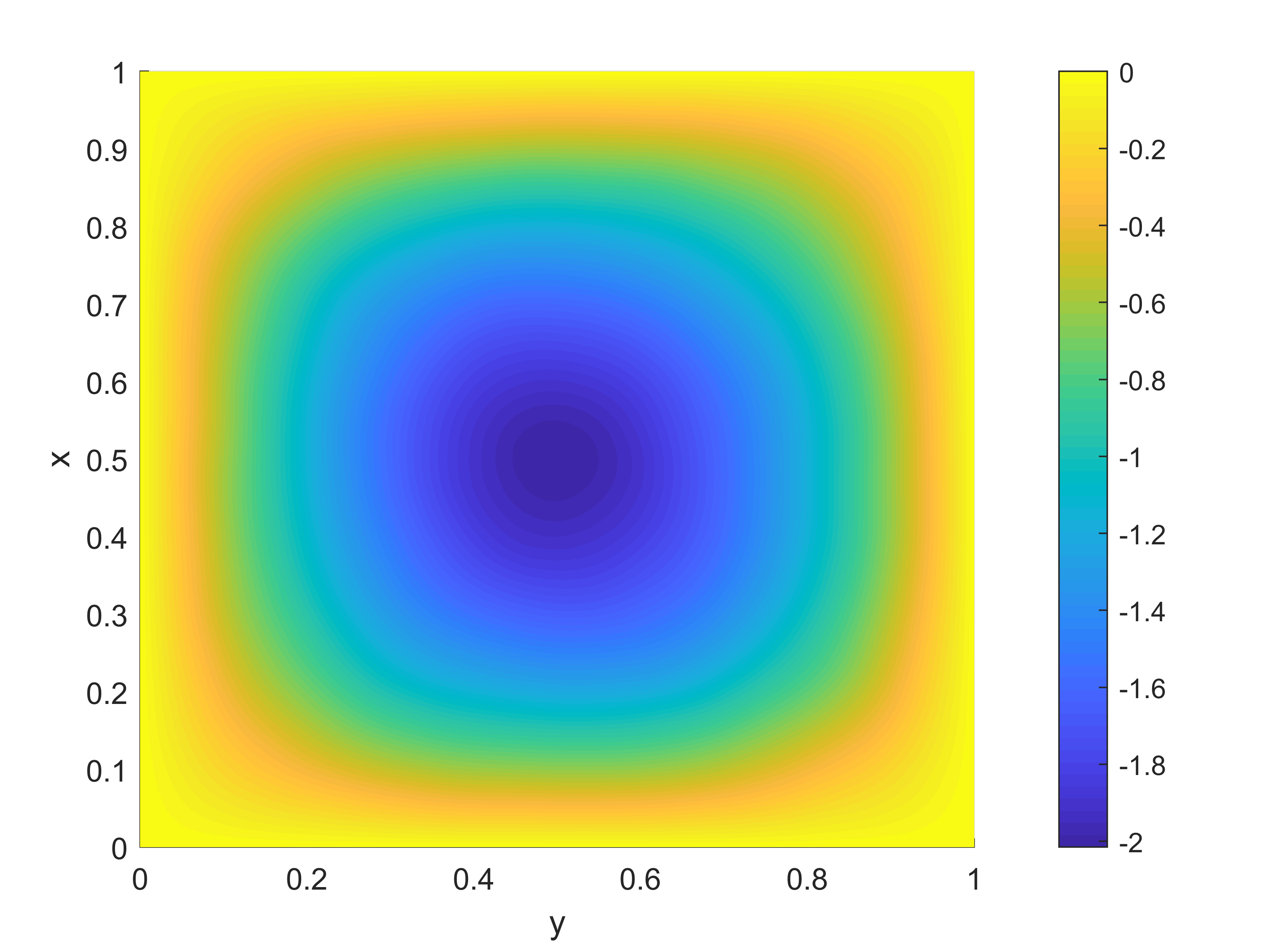}
		\caption{$a_1^{rec}$, $(\alpha,\rho)=(1.2,3\times10^{-5})$}
		\label{figure7}
	\end{minipage}
	\begin{minipage}{0.49\linewidth}
		\centering
		\includegraphics[width=1\linewidth]{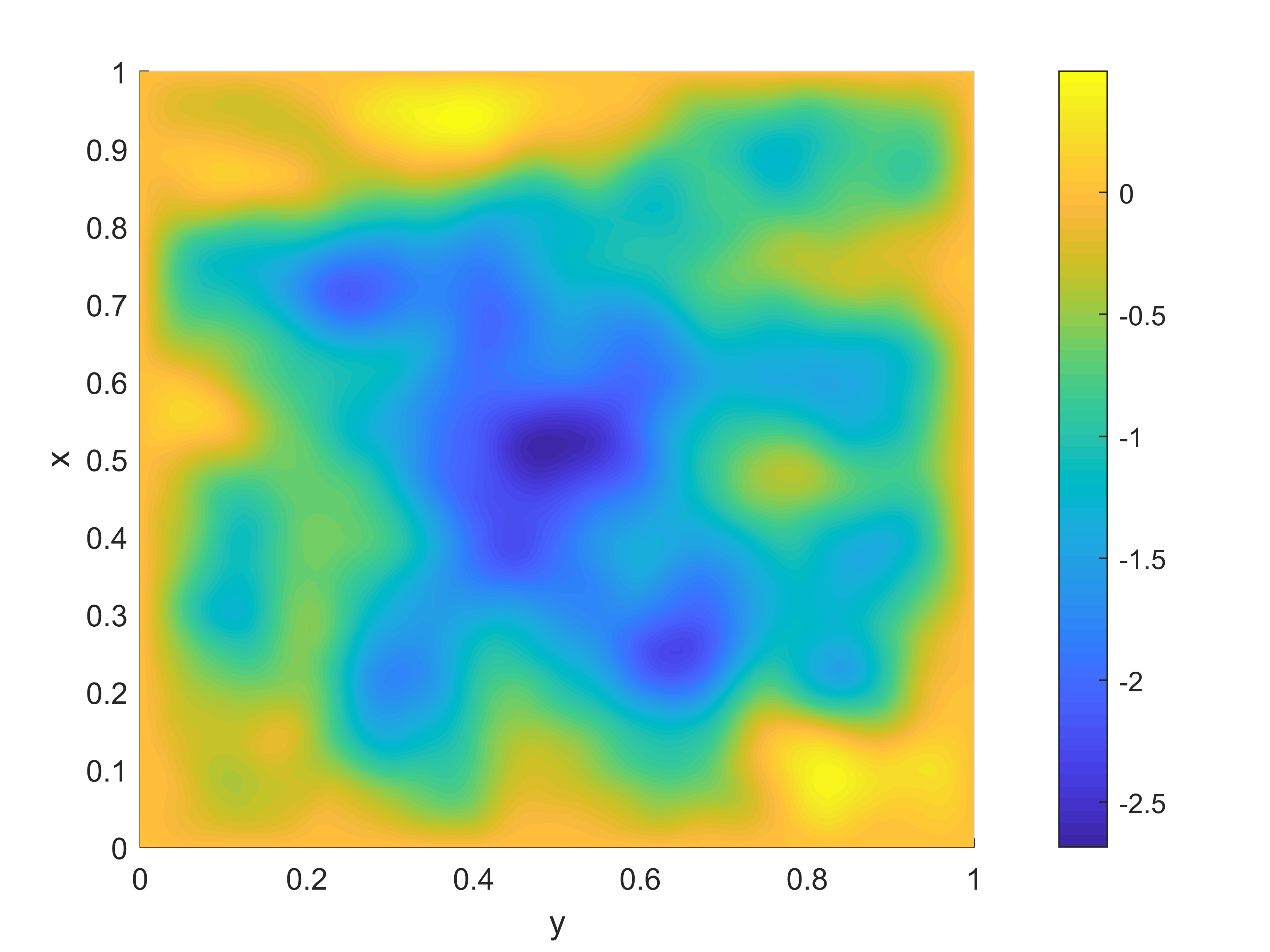}
		\caption{$a_1^{rec}$, $(\alpha,\rho)=(1.2,10^{-6})$}
		\label{figure8}
	\end{minipage}
        \vspace{0.3cm}
        \vfill
        \begin{minipage}{0.49\linewidth}
		\centering
		\setlength{\abovecaptionskip}{0.28cm}
		\includegraphics[width=1\linewidth]{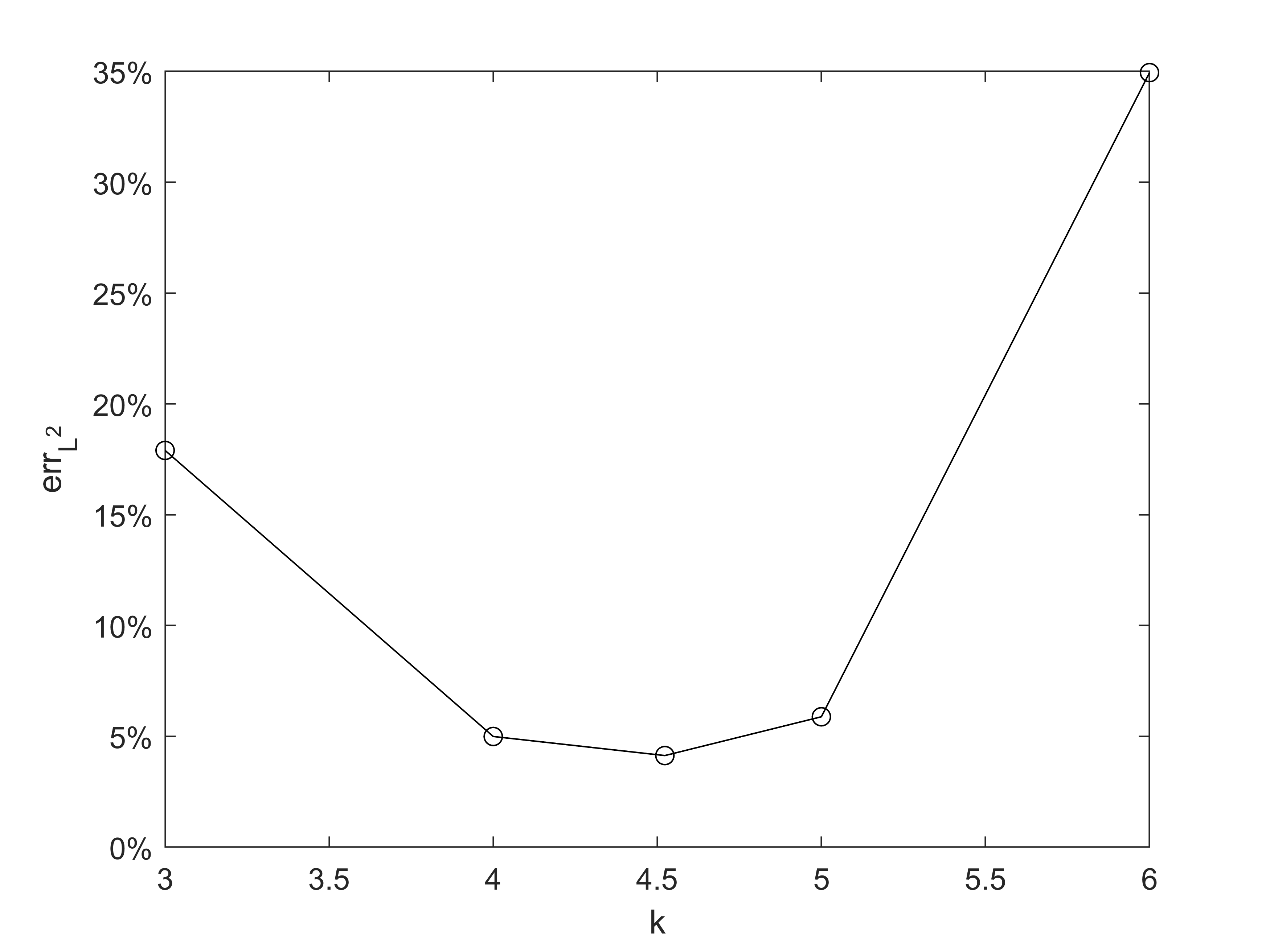}
		\caption{$(10^{-k},err_{L^2})$, $\alpha=1.2$}
		\label{figure_add1}
	\end{minipage}
	\begin{minipage}{0.49\linewidth}
		\centering
		\includegraphics[width=1\linewidth]{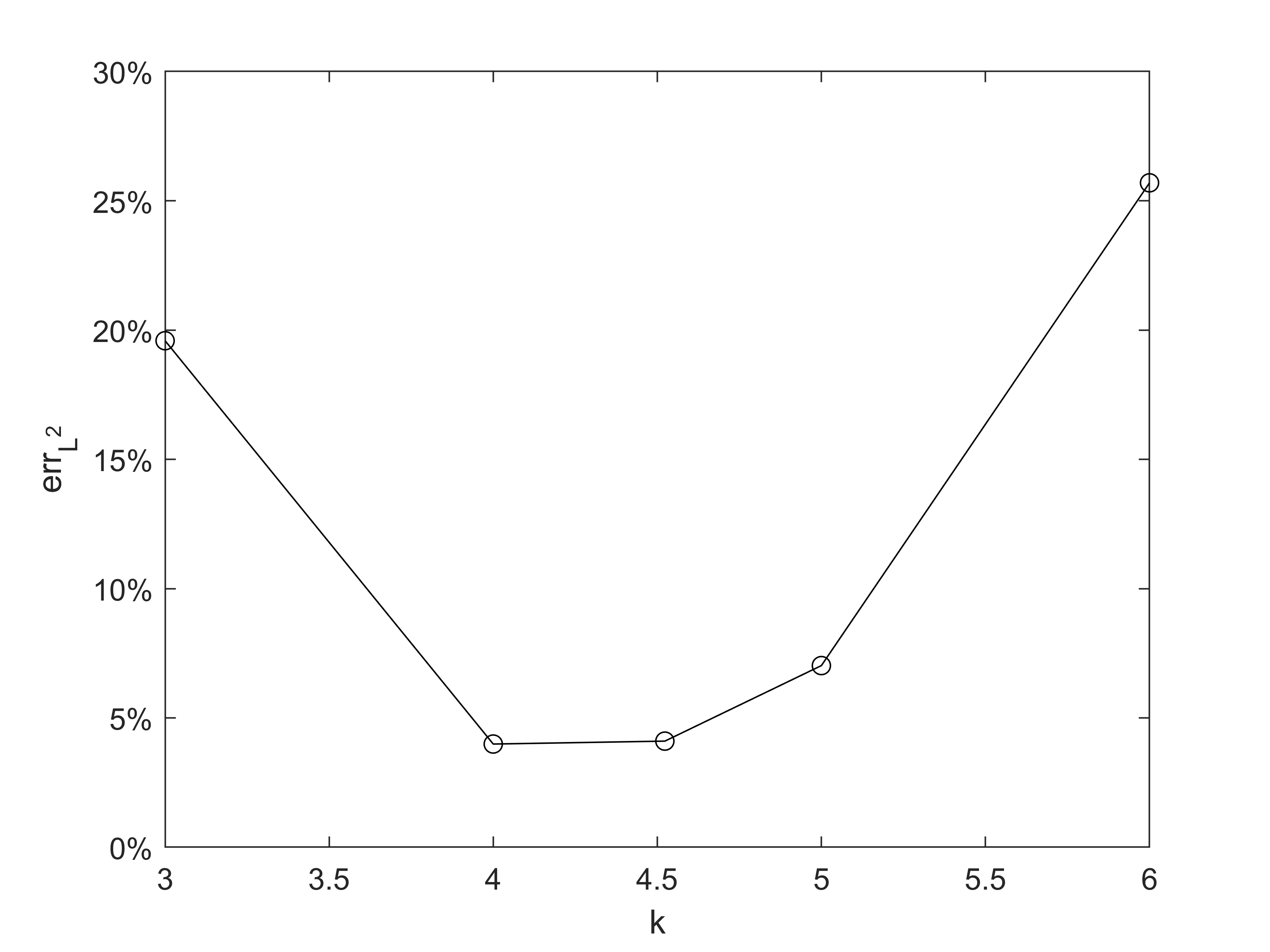}
		\caption{$(10^{-k},err_{L^2})$, $\alpha=1.8$}
		\label{figure_add2}
	\end{minipage}
\end{figure}

\begin{example}\label{ex-sigma=0.4}
We fix the noise parameter $\sigma=0.4$, the initial guess $a_1^0=0$, and the initial function is as follows:
\begin{align*}
a_1=7x^{0.75}(x-1)\sin(2\pi y),~T=1.
\end{align*}
\end{example}

Unlike Example \ref{ex-sigma=0.2}, the initial function $a_1$ has a weak singularity at $x=0$.
In Example \ref{ex-sigma=0.4}, the maximum norm of $Sa_1$ is $\|Sa_1\|_{L^\infty} \approx 0.43$, $0.61$ for $\alpha=1.2$, $1.8$. Thus, noise level $\frac{\sigma}{\|Sa_1\|_{L^\infty}}\approx93\%$, $66\%$ for $\alpha=1.2$, $1.8$.
Based on Remark \ref{remark-2}, the optimal regularization parameter $\rho = O\big((\sigma n^{-\frac12} \|a_1\|^{-1}_{L^2(\Omega)})^{4/3}\big)=O( 7.51\times 10^{-4})$ for $L^2$ regularization.
Then, we consider the scattered point measurements-based $H^1$ regularization approach below. In Remark \ref{remark-1}, it gives the optimal regularization parameter for $H^1$ regularization $\rho=O\left((\sigma n^{-\frac12} \|a_1\|^{-1}_{H^1(\Omega)})^{3/2}\right)=O( 1.60\times 10^{-5})$.

\begin{algorithm}
In this part, we design an iterative algorithm to find the optimal parameter $\rho$. From Remarks \ref{remark-1} and \ref{remark-2}, we know that the optimal parameter $\rho$ depends on $\sigma$ and $\|a_1^*\|_{X}$. In fact, both the noise parameter $\sigma$ and the norm of $a_1^*$ are unknown. A direct idea is to find $\rho$ by using a fixed-point iteration method. At the beginning, an initial guess is given $\rho_1=n^{-\frac{4(1+\beta)}{4(1+\beta)+d}}$ and then we obtain $a_1^1$. When $a_1^k$ the $k$-th step is available, we update $\rho$ as,
$$\rho_k^{1/2+d/8(1+\beta)}=n^{-1/2}\|a_1^k\|_{X}^{-1}\|S_{\tau,h}a_1^k-m\|_n,$$
where $S_{\tau,h}$ means the discrete approximation of forward operator $S$.

\begin{table}[!ht]
\caption{Optimal parameter iterative algorithm.}
\renewcommand{\arraystretch}{1}
\def\temptablewidth{1.0\textwidth}
\begin{center}
 \begin{tabular*}{\temptablewidth}{@{\extracolsep{\fill}}l}\hline
 \textbf{An algorithm of finding optimal parameter and recovering initial function.}\\ \hline
 1. \textbf{Input:} Observation data $m$, number of observation data $n$, error threshold of $\rho$ $tol_{\rho}$.  \\
 2. Setup initial guess $\rho_1\leftarrow n^{-\frac{4(1+\beta)}{4(1+\beta)+d}}$;   \\
 3. Solve $a_1^1$ from $m$ with parameter $\rho=\rho_1$; \\
 4. $\rho_2\leftarrow (n^{-1/2}\|S_{\tau,h}a_1^1-m\|_n\|a_1^1\|_{X}^{-1})^{\frac{8(1+\beta)}{4(1+\beta)+d}}$; \\
 5. $k \leftarrow 2$; \\
 6. \textbf{while} $|\rho_k-\rho_{k-1}|>tol_{\rho}$ \textbf{do}\\
 ~~~ 6.1 Solve $a_1^{k}$ from $m$ with parameter $\rho=\rho_k$;\\
 ~~~ 6.2 $\rho_{k+1}\leftarrow (n^{-1/2}\|S_{\tau,h}a_1^k-m\|_n\|a_1^k\|_{X}^{-1})^{\frac{8(1+\beta)}{4(1+\beta)+d}}$; \\
 ~~~ 6.3 $k \leftarrow k+1$;\\
 ~~~ \textbf{end}\\
 7. \textbf{Output:} parameter $\rho_k$ and solved initial function $a_1^{k}$.
 \\ \hline
\end{tabular*}
\end{center}
\end{table}

\end{algorithm}

We consider using the iterative algorithm of optimal parameters to recover the initial function $a_1$ in Example \ref{ex-sigma=0.4}. The algorithm enjoys the monotone property like the result in
\cite{Jin2024arXiv}.
The iteration of $\rho$ will stop when $|\rho_{k+1}-\rho_k|<tol_{\rho}=10^{-6},10^{-8}$ is used for the regularization of $L^2$ and $H^1$, respectively.
As shown in Figure \ref{figure_add3}, $\|S_{\tau,h}a_1^k-m\|_n$ is very close to the noise level parameter $\sigma=0.4$ after only 4 iterations. In the fourth iteration, the regularization parameters found for the regularization of $L^2$ are $6.8804\times10^{-4}$, $7.0666\times10^{-4}$ for $\alpha=1.2,1.8$, respectively. Similarly, for $H^1$ regularization, it shows that the regularization parameters found are $8.4303\times10^{-6}$, $1.1138\times10^{-5}$ for $\alpha=1.2,1.8$, respectively, see Figure \ref{figure_add5}.
The parameter $\rho$ found by the algorithm is very close to the optimal one.
This further demonstrates the practicality of the algorithm. It does not rely on any unknown information in actual situations, such as noise levels, initial value functions, etc. The reconstruction results at the end of the iterations are given in Figures \ref{figure9}-\ref{figure12}. And $err_{L^2}$ and $err_{H^{-1}}$ of them are $13.185\%$ and $8.899\%$, respectively.

\begin{figure}[!htb]
        \centering
	\begin{minipage}{0.49\linewidth}
		\centering
		\setlength{\abovecaptionskip}{0.28cm}
		\includegraphics[width=1\linewidth]{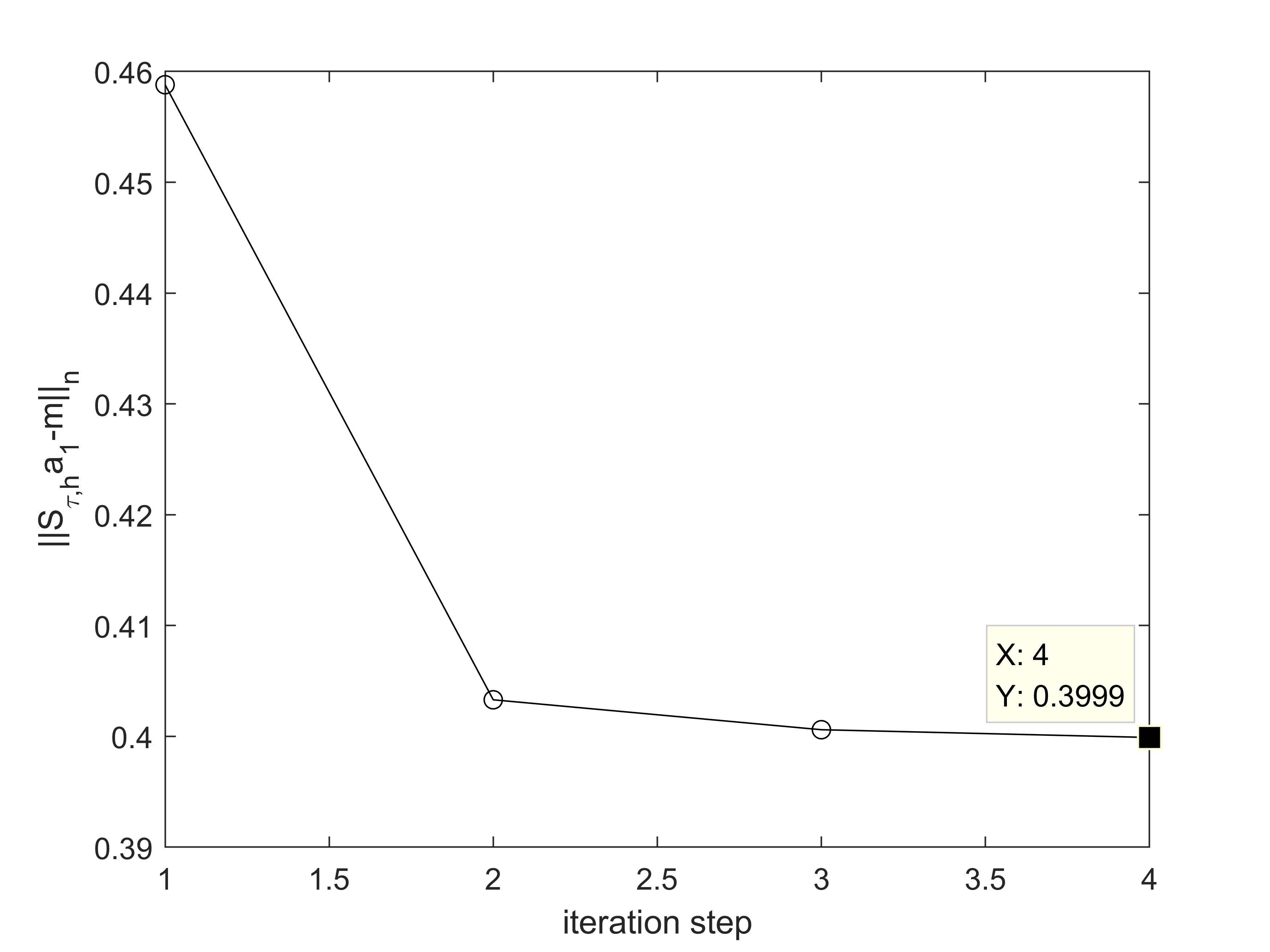}
	\end{minipage}
	\begin{minipage}{0.49\linewidth}
		\centering
		\includegraphics[width=1\linewidth]{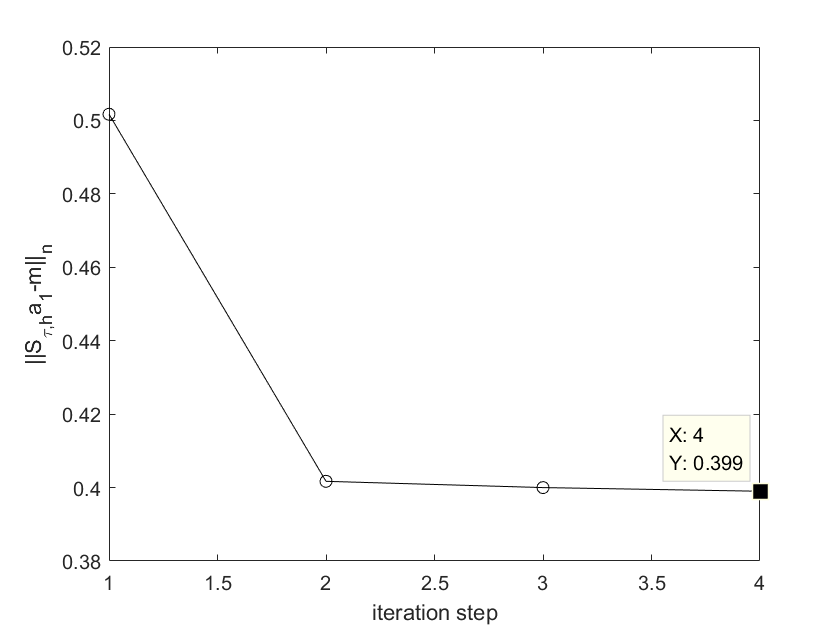}
	\end{minipage}
        \caption{The relative error $\|S_{\tau,h}a_1-m\|_n$ at each iteration in $L^2$ REG., $\alpha=1.2$(left) and $\alpha=1.8$(right).}
        \label{figure_add3}
\end{figure}

\begin{figure}[!htb]
        \begin{minipage}{0.49\linewidth}
		\centering
		\setlength{\abovecaptionskip}{0.28cm}
		\includegraphics[width=1\linewidth]{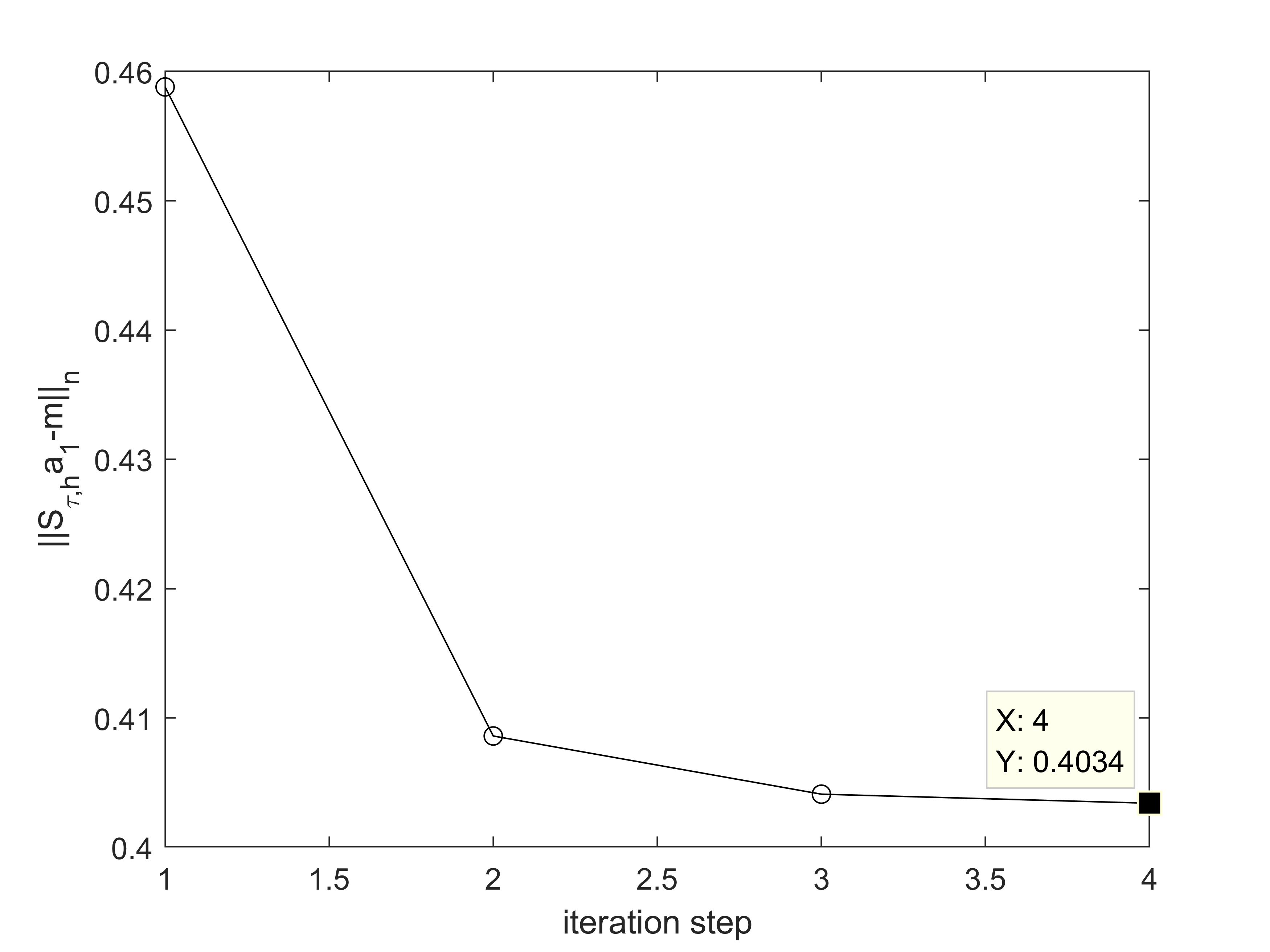}
		
	\end{minipage}
	\begin{minipage}{0.49\linewidth}
		\centering
		\includegraphics[width=1\linewidth]{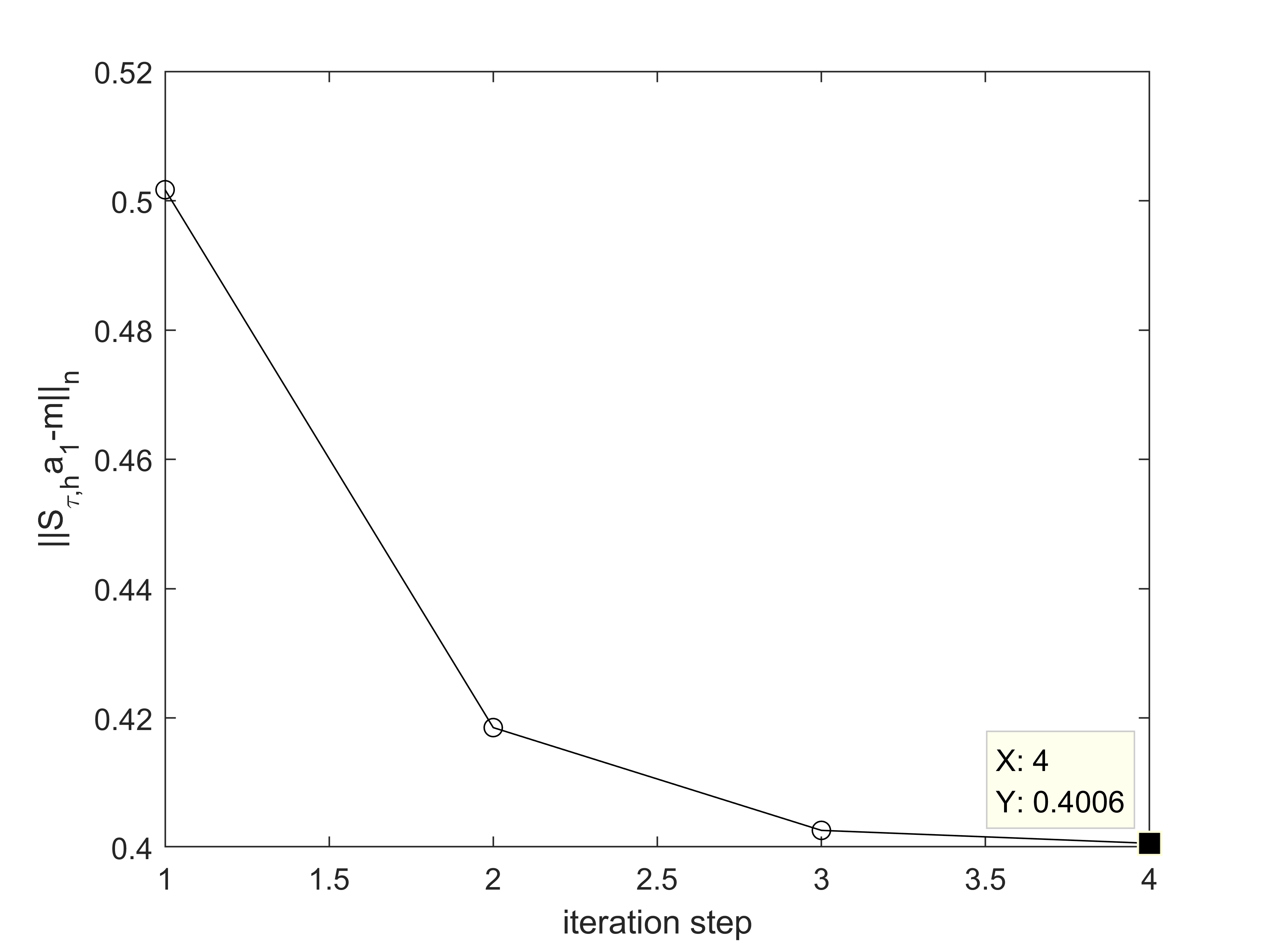}
	\end{minipage}
        \caption{The relative error $\|S_{\tau,h}a_1-m\|_n$ at each iteration in $H^1$ REG., $\alpha=1.2$(left) and $\alpha=1.8$(right).}
        \label{figure_add5}
\end{figure}

\begin{figure}[!htb]
        \begin{minipage}{0.49\linewidth}
		\centering
		\setlength{\abovecaptionskip}{0.28cm}%
		\includegraphics[width=1\linewidth]{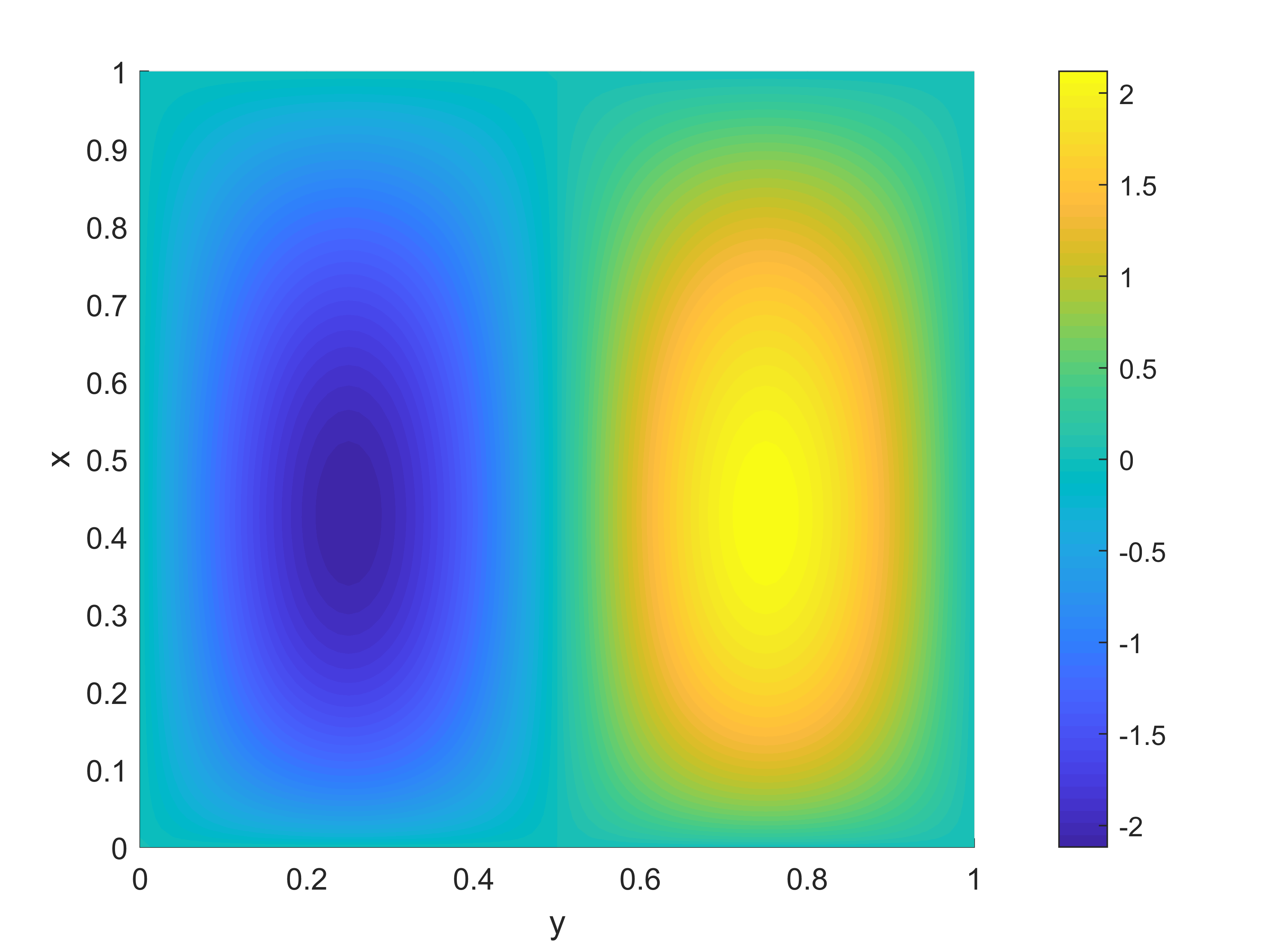}
		\caption{True $a_1$ in Ex. \ref{ex-sigma=0.4}.}
	\end{minipage}
	\begin{minipage}{0.49\linewidth}
		\centering
		\includegraphics[width=1\linewidth]{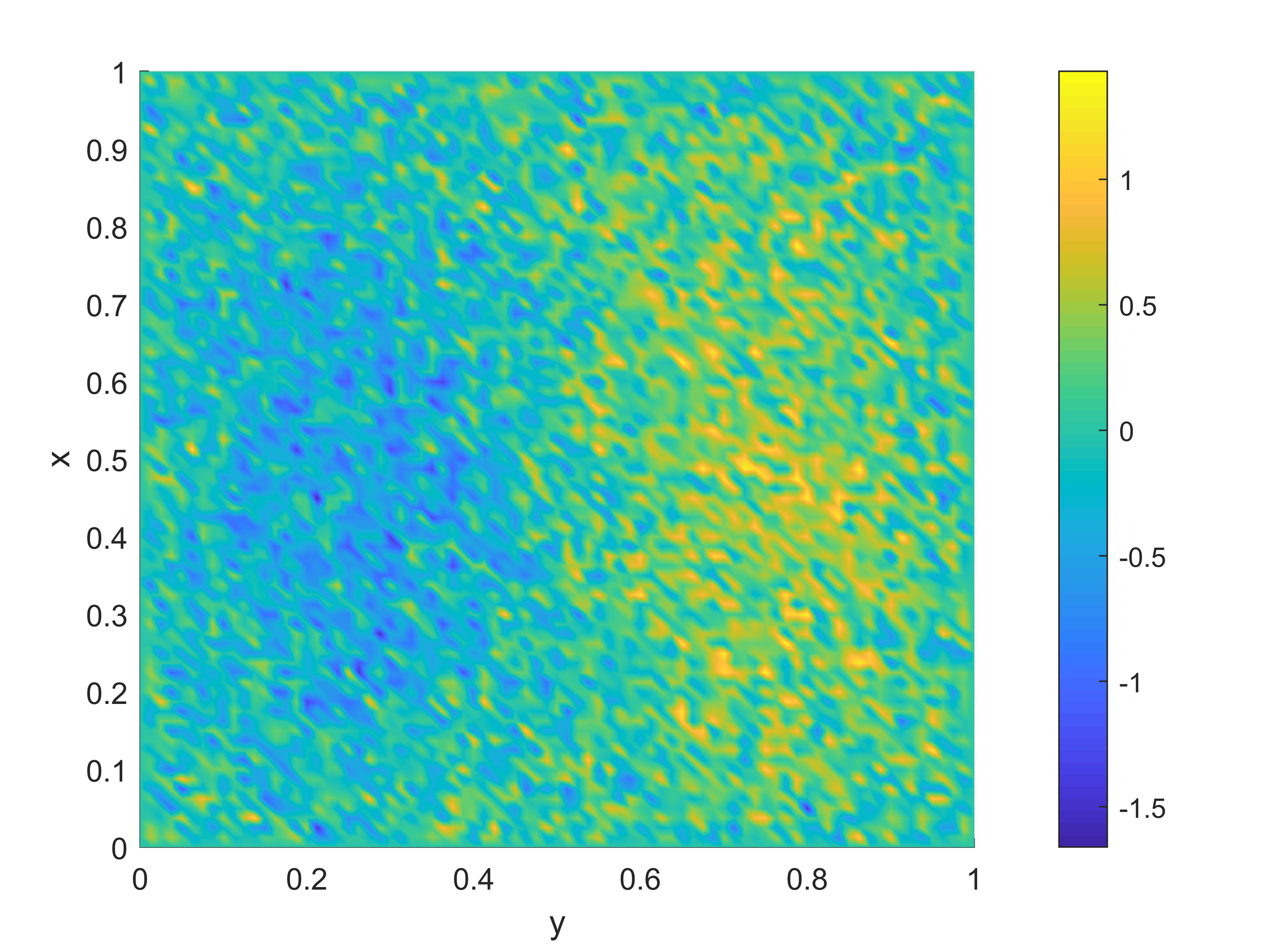}
		\caption{measurements, $\alpha=1.2$}
	\end{minipage}
\end{figure}

\begin{figure}[!htb]
        \begin{minipage}{0.49\linewidth}
		\centering
		\setlength{\abovecaptionskip}{0.28cm}
		\includegraphics[width=1\linewidth]{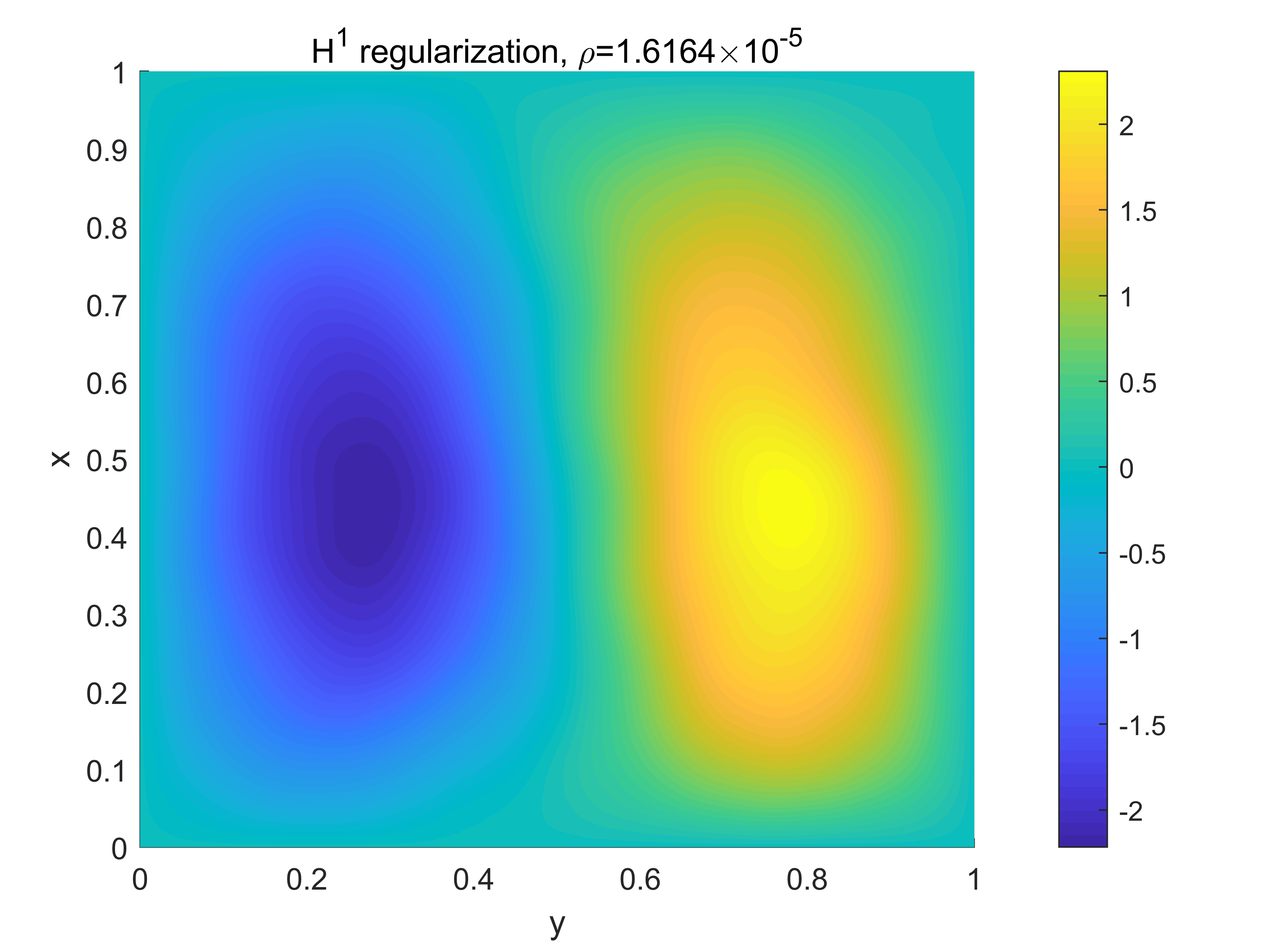}
		\caption{$a_1^{rec}$, $\alpha=1.2$ ($H^1$ REG.)}
		\label{figure9}
	  \end{minipage}
        \begin{minipage}{0.49\linewidth}
		\centering
		\setlength{\abovecaptionskip}{0.28cm}
		\includegraphics[width=1\linewidth]{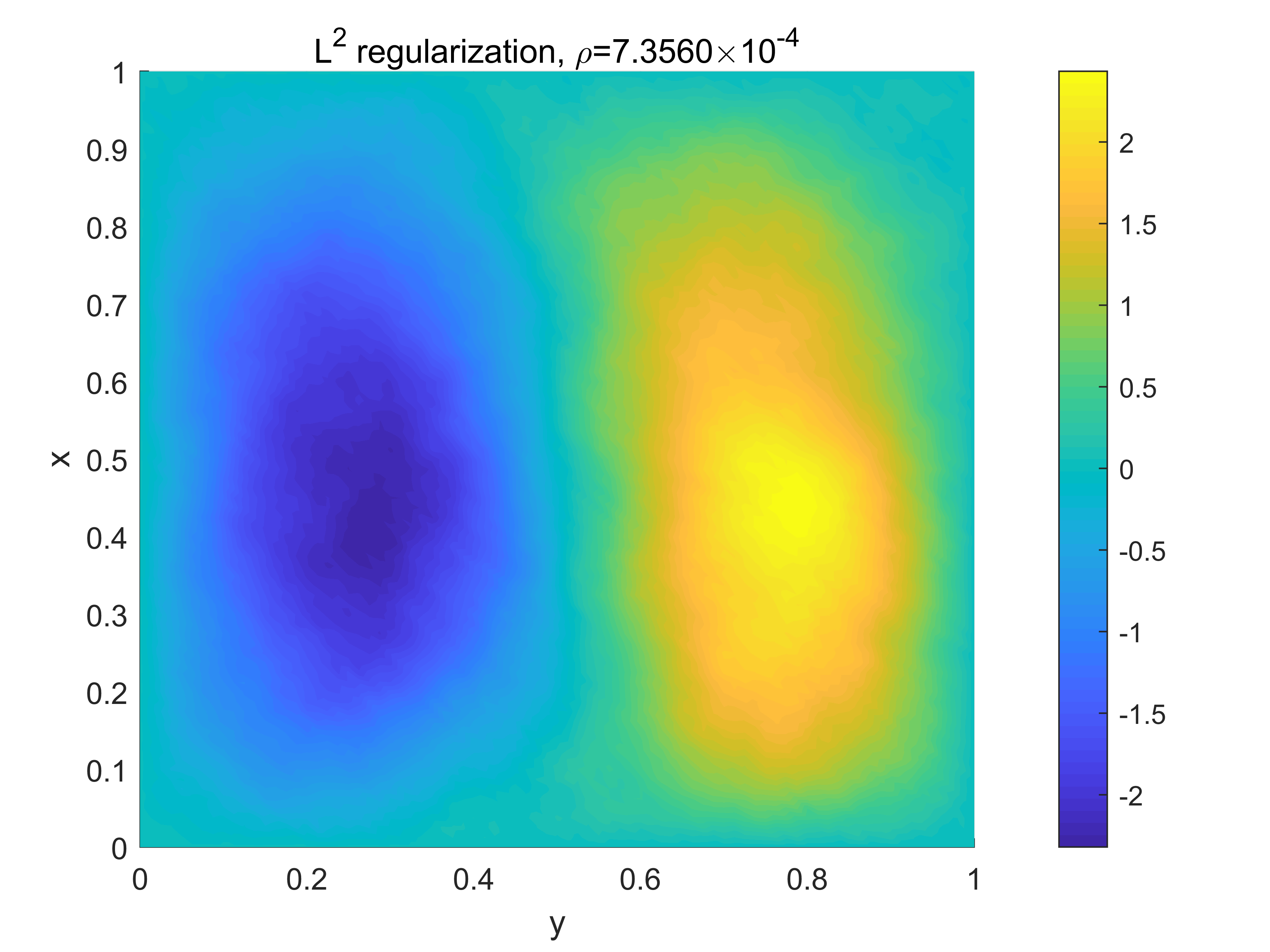}
		\caption{$a_1^{rec}$, $\alpha=1.2$ ($L^2$ REG.)}
		\label{figure12}
	  \end{minipage}
\end{figure}

\begin{example}\label{ex-sigma=0.01}
We try to recover a piecewise constant initial condition by using the iterative algorithm. Let the noise parameter $\sigma=0.01$, the initial guess $a_1^0=0$, and the initial function is as follows:
\begin{align*}
a_1=\chi_{[0.25,0.75]\times[0.25,0.75]}(x,y),~~T=0.1.
\end{align*}
\end{example}
The optimal regularization parameter $\rho = O\big((\sigma n^{-\frac12} \|a_1\|^{-1}_{L^2(\Omega)})^{4/3}\big)=O( 1.57\times 10^{-5})$ for $L^2$ regularization based on Remark \ref{remark-2}.  The iteration of $\rho$ will stop when $|\rho_{k+1}-\rho_k|<tol_{\rho}=10^{-7}$. The relative $H^{-1}$-norm error $err_{H^{-1}}$ between the true initial function $a_1$ and the reconstructed one $a_1^{rec}$ are $3.661\%$ and $1.918\%$ for $\alpha=1.2,1.8$, respectively. The reconstruction results are given in Figure \ref{ex3_figure2} and Figure \ref{ex3_figure4}. On the other hand, our algorithm is efficient. Figure \ref{ex3_figure5} shows that it takes only a few iterations to find the regularization parameters $\rho$.

\begin{figure}[!htb]
        \begin{minipage}{0.49\linewidth}
		\centering
		\setlength{\abovecaptionskip}{0.28cm}
		\includegraphics[width=1\linewidth]{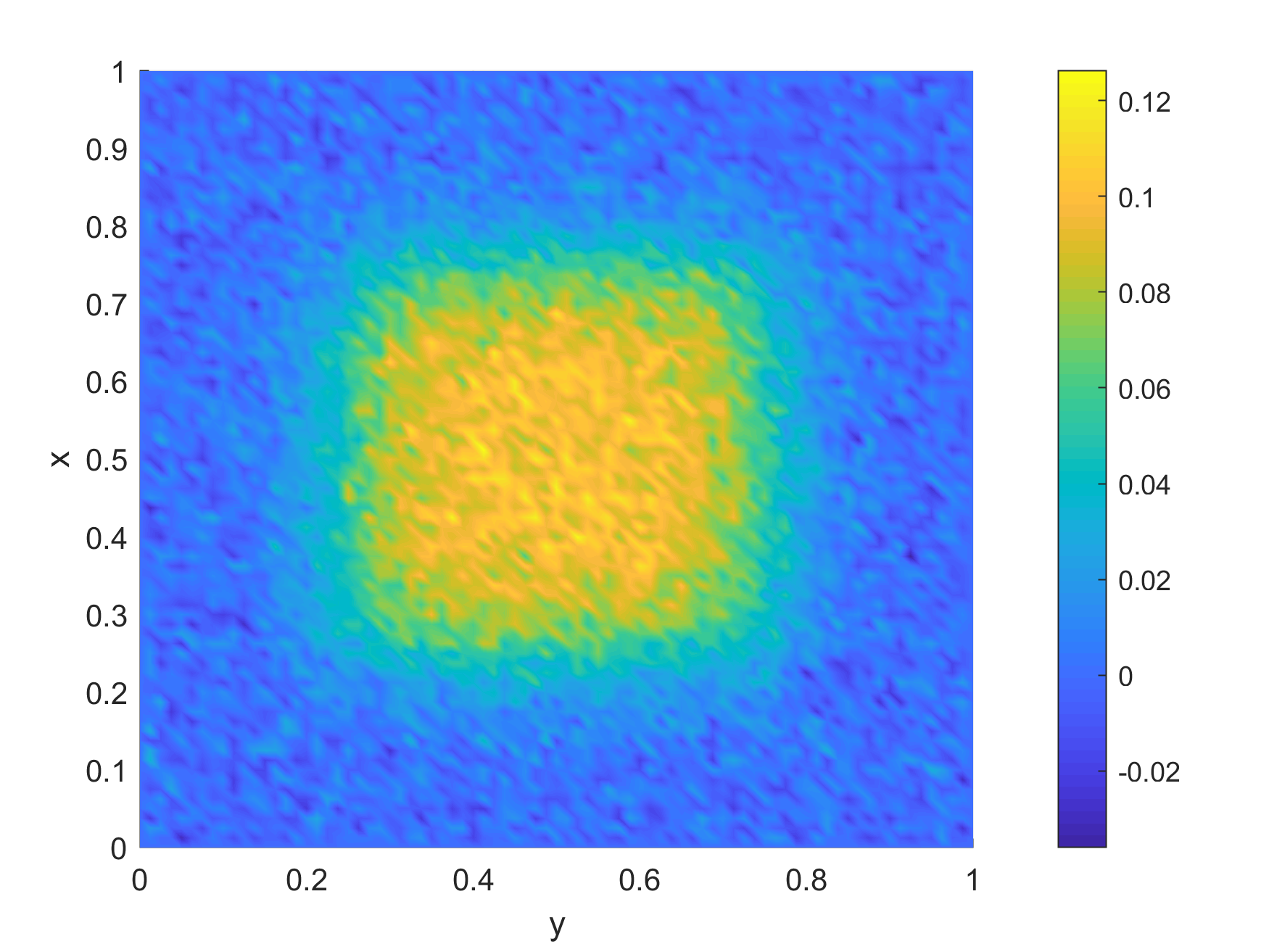}
		\caption{measurements, $\alpha=1.2$}
		\label{ex3_figure1}
	  \end{minipage}
        \begin{minipage}{0.49\linewidth}
		\centering
		\setlength{\abovecaptionskip}{0.28cm}
		\includegraphics[width=1\linewidth]{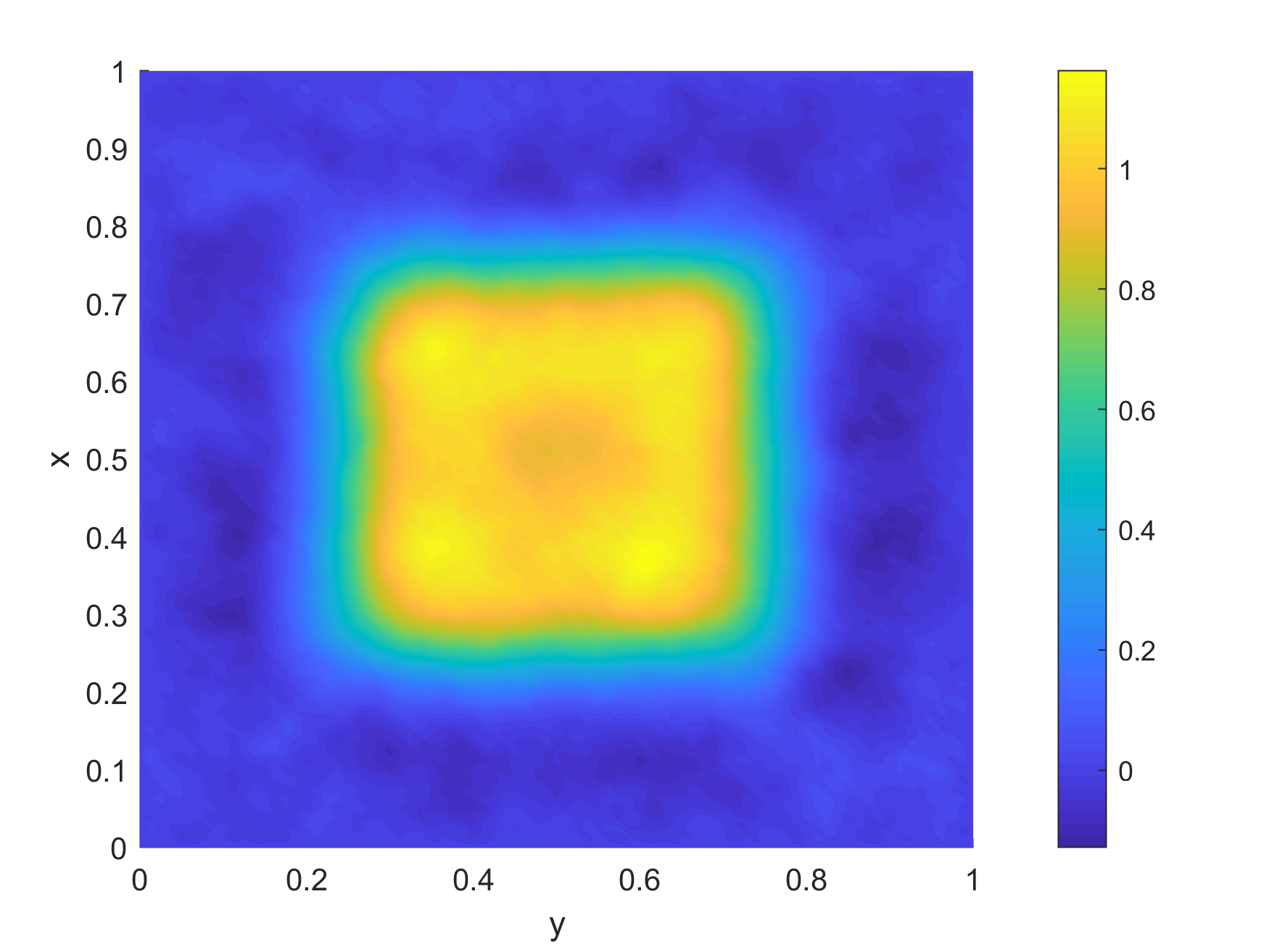}
		\caption{$a_1^{rec}$, $(\alpha,\rho)=(1.2,1.69\times10^{-5})$}
		\label{ex3_figure2}
	  \end{minipage}
\end{figure}

\begin{figure}[!htb]
        \begin{minipage}{0.49\linewidth}
		\centering
		\setlength{\abovecaptionskip}{0.28cm}
		\includegraphics[width=1\linewidth]{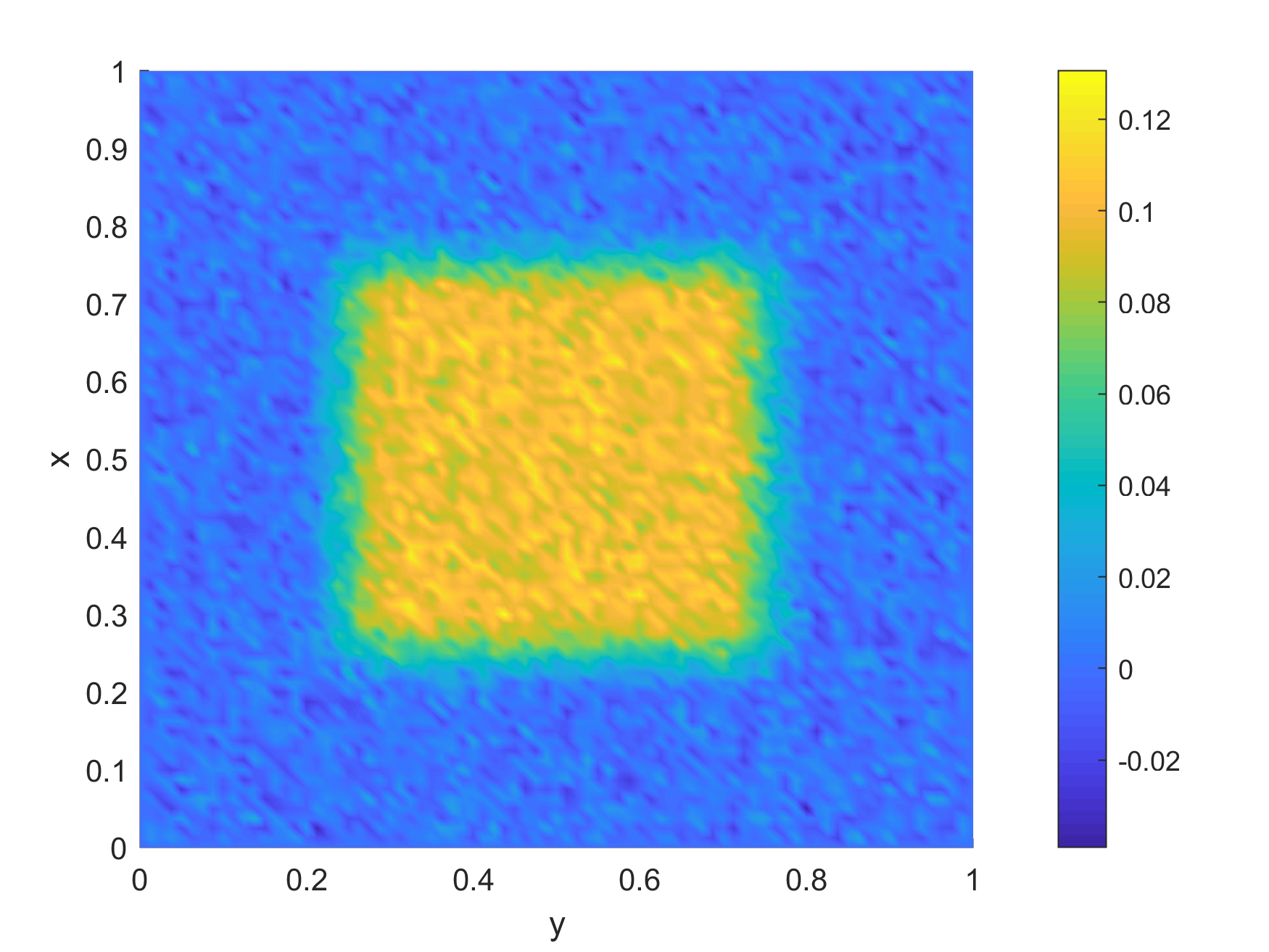}
		\caption{measurements, $\alpha=1.8$}
		\label{ex3_figure3}
	  \end{minipage}
        \begin{minipage}{0.49\linewidth}
		\centering
		\setlength{\abovecaptionskip}{0.28cm}
		\includegraphics[width=1\linewidth]{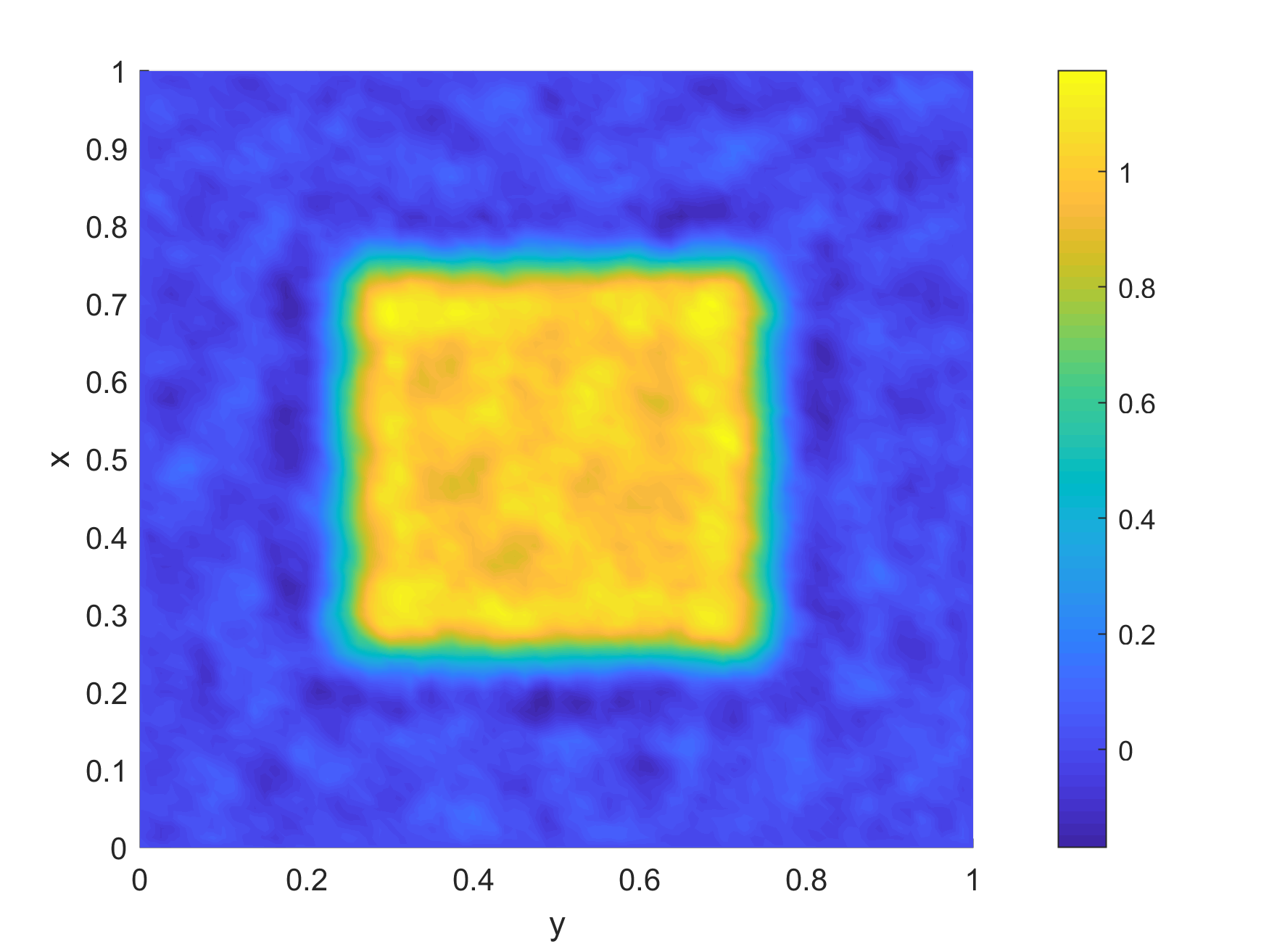}
		\caption{$a_1^{rec}$, $(\alpha,\rho)=(1.8,1.58\times10^{-5})$}
		\label{ex3_figure4}
	  \end{minipage}
\end{figure}

\begin{figure}[!htb]
        \begin{minipage}{0.49\linewidth}
		\centering
		\setlength{\abovecaptionskip}{0.28cm}
		\includegraphics[width=1\linewidth]{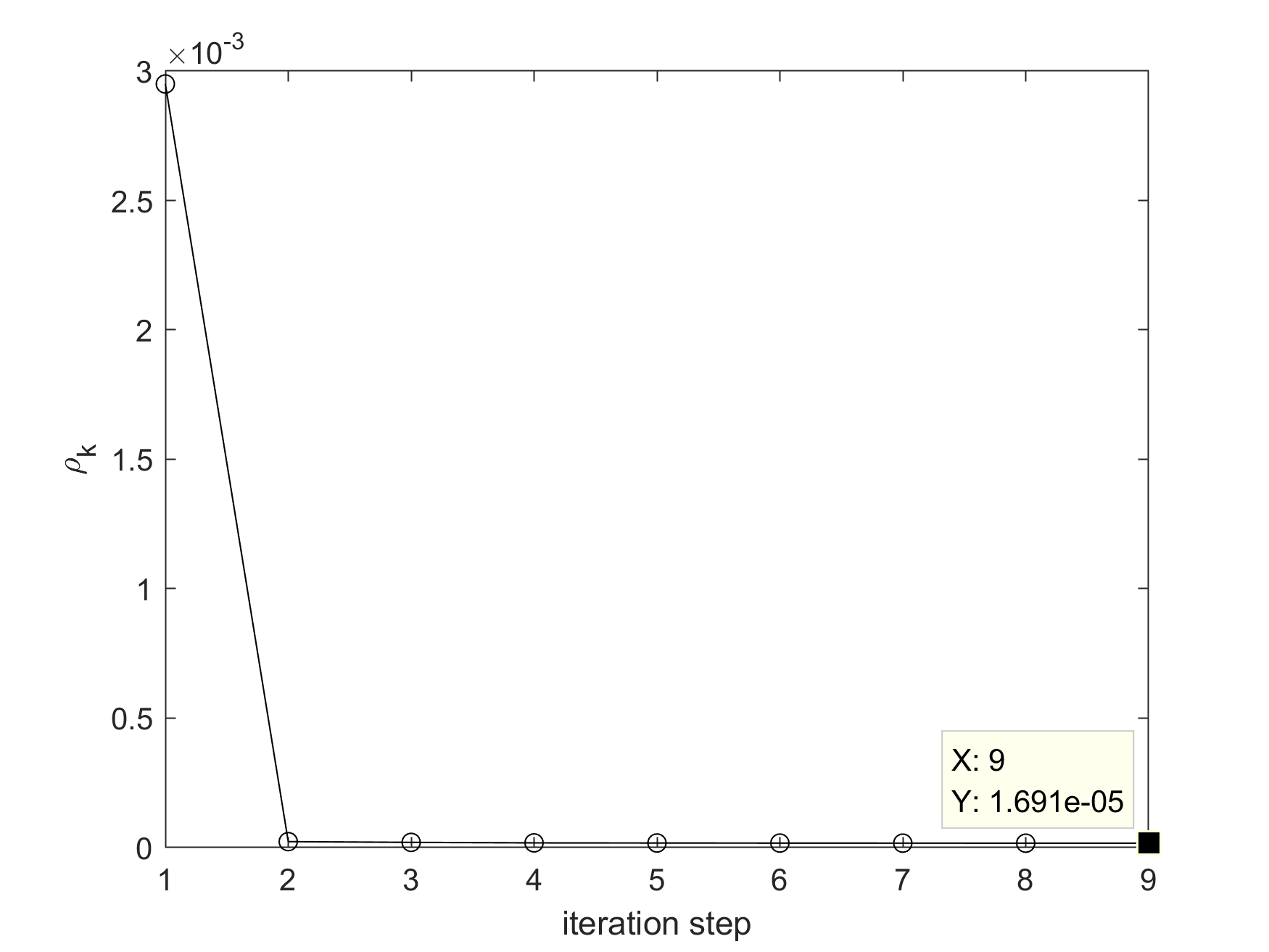}
	  \end{minipage}
        \begin{minipage}{0.49\linewidth}
		\centering
		\setlength{\abovecaptionskip}{0.28cm}
		\includegraphics[width=1\linewidth]{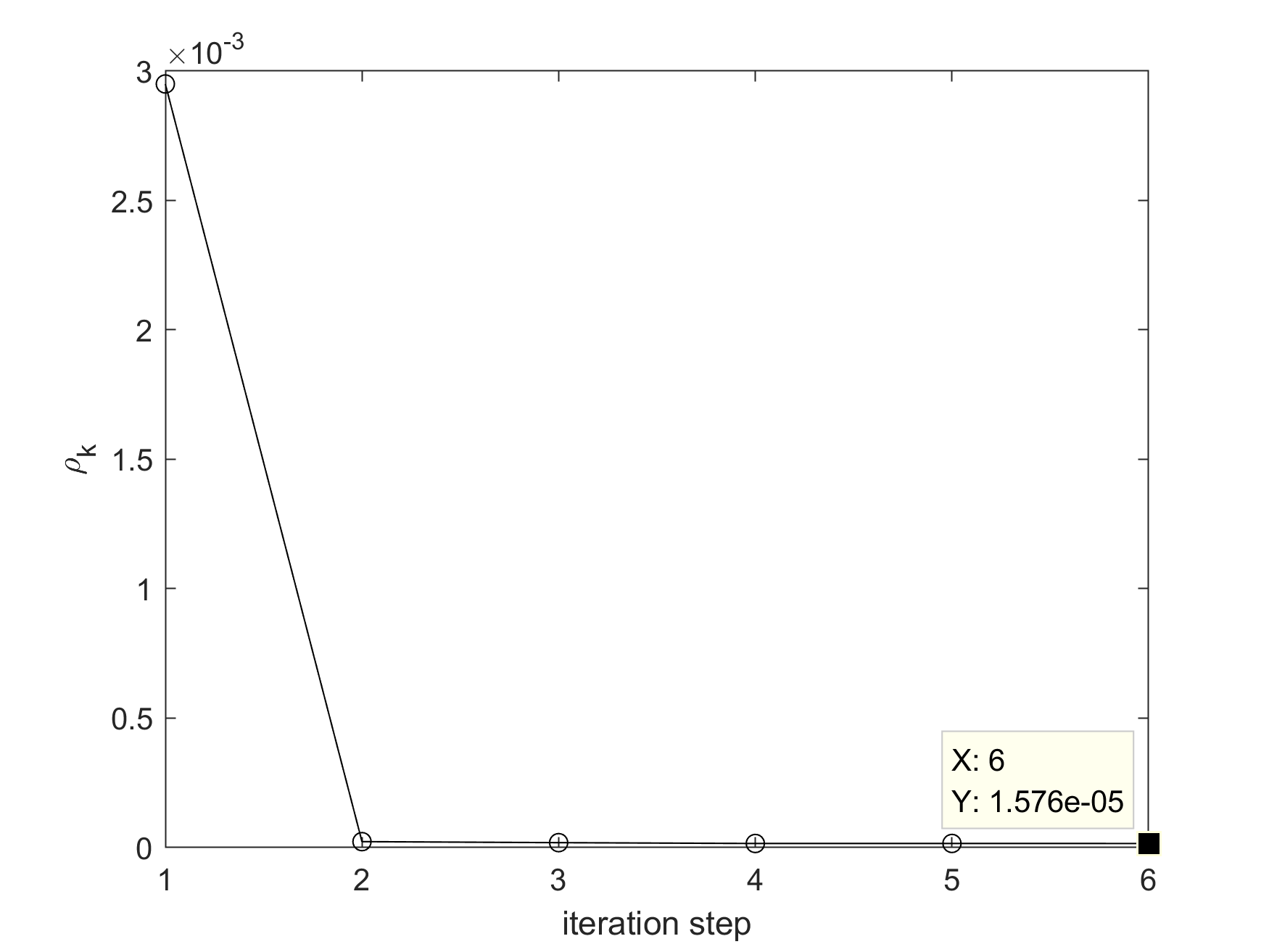}
	  \end{minipage}
        \caption{The regularization
parameter $\rho_k$ at each iteration, $\alpha=1.2$ (left) and $\alpha=1.8$ (right).}
		\label{ex3_figure5}
\end{figure}

\section{Concluding remarks}
This paper discussed the backward problem for fractional wave equations. Using the properties of the Mittag-Leffler functions, we established the stability results for the backward problem. In particular, additional constraints on the terminal time $T$ are not required when $\alpha$ lies in $(1,\frac43]$, which broadens the conditions under which the stability of solutions to backward-in-time problems can be ensured.

Moreover, addressing the inconsistency between theoretical treatment and numerical implementations in the backward problems for fractional wave equations, we introduced the theory of Tikhonov regularization based on scattered point measurements. We obtained stochastic convergence and provided the selection of regularization parameters as well as optimal error estimates. It should be noted that our numerical analysis and experiments indicate that despite the presence of large observation errors, we can still obtain more precise inversion results by increasing the number of observation points, which is difficult for classical regularization algorithms to achieve. %Our method demonstrates greater robustness and flexibility, especially in handling data with high noise levels, making it particularly valuable in practical applications.

Currently, we have only recovered one initial value. In the future, we plan to extend our method to recover two initial values. However, this may require terminal values at two time levels to achieve an accurate inversion of both initial conditions.

\section*{Declarations}
On behalf of all authors, the corresponding author states that there is no conflict of interest. No datasets were generated or analyzed during the current study.

%\section*{Acknowledgments}
%Zhiyuan Li thanks the National Natural Science Foundation of China (no. 12271277) and Ningbo Youth Leading Talent Project (no. 2024QL045). The research of Wenlong Zhang is supported by the National Natural Science Foundation of China under grant numbers No.12371423 and No.12241104. This work is partially supported by the Open Research Fund of the Key Laboratory of Nonlinear Analysis \& Applications (Central China Normal University), Ministry of Education, China (no. NAA20230RG002).

\end{document}